\documentclass[10pt,a4paper,twoside]{article}
\usepackage[utf8]{inputenc}
\usepackage[english]{babel}

\usepackage[pdfencoding=unicode, psdextra]{hyperref}
\usepackage{eucal,amsmath, amssymb, amsthm}
\usepackage{textcomp, csquotes, subcaption}
\usepackage{relsize}
\usepackage{tikz}
\usepackage[capitalise]{cleveref}
\usetikzlibrary{positioning,shapes,fit,arrows}

\usepackage{pgf,pgfplots}
\usepackage{outlines}
\usepackage{float}
\usepackage{comment}
\usepackage{cancel}
\usepackage{diagbox}
\usepackage{qtree}
\usepackage{fullpage}
\usepackage[bottom]{footmisc}

\newtheorem{theorem}{Theorem}[section]
\newtheorem{conjecture}[theorem]{Conjecture}
\newtheorem{corollary}[theorem] {Corollary}
\newtheorem{definition}[theorem]{Definition}
\newtheorem{example}[theorem]{Example}
\newtheorem{lemma}[theorem]{Lemma}

\newtheorem{observation}[theorem]{Observation}

\newtheorem{proposition}[theorem]{Proposition}

\newtheorem{claim}[theorem]{Claim}
\newtheorem{question}{Question}
\newcommand{\Z}{\mathbb{Z}}
\newcommand{\K}{\mathcal{K}}
\newcommand{\cF}{\mathcal{F}}
\DeclareMathOperator{\Ind}{Ind}

\vspace{5cm}

\pgfplotsset{compat=1.16}  
  
\begin{document}
  
  \label{'ubf'}  
\setcounter{page}{1}                                 %Put here the starting page number

\markboth {\hspace*{-15mm} \centerline{\footnotesize \sc
         % Put here the left page top label 
   Higher independence complexes of graphs and their homotopy types \\}
                 }
                { \centerline                           {\footnotesize \sc  
                   %put here the author's name
         Priyavrat Deshpande and Anurag Singh                                                 } \hspace*{-9mm}              
               }

%\vspace*{-2cm}
%\begin{flushleft}
%    {\footnotesize\it J. of the Ramanujan Mathematical Society,  XX, No. XX (XXXX),  pp. XX--XX \\ %\pageref{'ubf'}-\pageref{'ubl'}\\  
%    % Use the appropriate labels
%    }
%           \vspace*{0.3cm}       %{2.2cm}
%\end{flushleft}

\begin{center}
{ 
       {\Large \textbf { \sc  Higher independence complexes of graphs and their homotopy types
    % Put the title of the paper here
                               }
       }
\\

\medskip

{\sc Priyavrat Deshpande }\\
{\footnotesize Chennai Mathematical Institute, India}\\

{\footnotesize e-mail: {\it pdeshpande@cmi.ac.in}}
}\\
{\sc Anurag Singh }\\
{\footnotesize Chennai Mathematical Institute, India}\\

{\footnotesize e-mail: {\it anuragsingh@cmi.ac.in}}\\
\end{center}

\thispagestyle{empty}

\hrulefill

\begin{abstract}  
{\footnotesize  For $r\geq 1$, the $r$-independence complex of a graph $G$ is a simplicial complex whose faces are subsets $I \subseteq V(G)$ such that each component of the induced subgraph $G[I]$ has at most $r$ vertices. In this article, we determine the homotopy type of $r$-independence complexes of certain families of graphs including complete $s$-partite graphs, fully whiskered graphs, cycle graphs and perfect $m$-ary trees. In each case, these complexes are either homotopic to a wedge of equidimensional spheres or are contractible. We also give a closed form formula for their homotopy types. 
}
 \end{abstract}
 \hrulefill



%\thanks{PD and AS are partially funded by a grant from Infosys Foundation. PD is also partially funded by the MATRICS grant MTR/2017/000239.}

\section{Introduction}
Let $G$ be a simple undirected graph. 
A subset $I \subseteq V(G)$ of the vertex set of G, is called an {\itshape independent set} if the vertices of $I$ are pairwise non-adjacent in $G$. 
The {\itshape independence complex} of $G$, denoted $\Ind_1(G)$, is a simplicial complex whose faces are the independent subsets of $V$.
The study of topology of independence complexes of graphs has received a lot of attention in last two decades.
For example, in Babson and Kozlov's  proof of Lov{\'a}sz's conjecture (in \cite{BK07}) regarding odd cycles and graph homomorphism complexes the independence complexes of cycle graphs played an important role. 
In \cite{Mesh03}, Meshulam related the homology groups of $\Ind_1(G)$ with the domination number of $G$. 
The problem of determining a closed form formula for the homotopy type of $\Ind_1(G)$ for various classes of graphs is also well studied. 
For instance, see \cite{Koz99} for paths and cycle graphs, \cite{Kaw10a} for forests, \cite{BLN08, BH17} for grid graphs, \cite{Kaw10b} for chordal graphs and \cite{SSA19} for categorical product of complete graphs and generalized mycielskian of complete graphs. 
Barmak \cite{Bar03} studied the topology of independence complexes of triangle-free graphs and claw-free graphs. 
He also gave a lower bound for the chromatic number of $G$ in terms of the strong Lusternik-Schnirelmann category of $\Ind_1(G)$.

Recently in \cite{PS18}, Paolini and Salvetti generalized the notion of independence complexes by defining $r$-independence complex for any $r\geq 1$. 
For a graph $G$, a subset $I \subseteq V(G)$ is called {\itshape $r$-independent} if each connected component of the induced subgraph $G[I]$ has at most $r$ vertices. 
For $r\geq 1$, the {\itshape $r$-independence complex} of $G$, denoted $\Ind_r(G)$ is a simplicial complex whose faces are all $r$-independent subsets of $V(G)$. 
They established a relationship between the twisted homology of the classical braid groups and the homology of higher independence complexes of associated Coxeter graphs. 
In particular, they showed that $r$-independence complexes of path graphs are homotopy equivalent to a wedge of spheres (see \cref{higher_ind_path}).

The notion of $r$-independent sets have also been explored from pure graph theoretic viewpoints. In particular, it relates to the notion of clustered graph coloring. We refer the reader to \cite{kang18, sampath93} and the dynamic survey of David Wood \cite{wood18}. The authors would like to thank Ross Kang for bringing these references to our attention. 

Other than those mentioned above, the topology of higher independence complexes can be used to study important combinatorial properties of graphs. For instance,

\begin{enumerate}
\item[(i)] The concept of distance domination number of graphs has been studied extensively by several authors; see, for example, \cite{Gri92,HMV07,TX09}. In a joint work with Shukla \cite{DSS19}, we have been able to establish a relation between the distance $r$-domination number of a graph $G$ and (homological) connectivity of $\Ind_r(G)$.

\item[(ii)]  In \cite{Bar03}, Barmak found a lower bound for the chromatic number of a graph $G$ in terms of strong Lusternik-Schnirelmann (LS) category of $\Ind_1(G)$. 
A natural generalization in this direction would be to find a relation between the clustered $r$-coloring number of a graph $G$ and the strong LS category of $\Ind_r(G)$. 
\end{enumerate}

The aim of this article is to initiate the study of these so-called \emph{higher independence complexes} of graphs.
Our focus is on determining a closed form formula for their homotopy type. 
In this article we identify several classes of graphs for which these complexes are either homotopic to a wedge of equi-dimensional spheres or are contractible. 
In each case we also determine the dimension of the spheres and their number; we achieve this using discrete Morse theory. 

The paper is organized as follows. In \cref{sec_prelim} we recall all the important definitions and relevant tools from discrete Morse theory. 
The formal definition and basic properties of higher independence complexes is given in \cref{sec_basic}. 
In the subsequent sections we have considered various graph classes and determined the homotopy type of $r$-independence complexes. 
As stated above, in all these instances we have shown that the homotopy type is that of wedge of spheres or the complex is contractible. 
%Whenever a complex has the homotopy type of wedge of spheres we have provided an explicit formula for the dimensions of the spheres involved and their number. 
The statements of all our results are long and involve combinatorially intricate expressions so we just list graph classes here and provide the reference. 
\begin{enumerate}
    \item The homotopy type of $\Ind_r(G)$ for $G$ a complete $s$-partite graph is given in \cref{theorem:complete r partite graph}.
    \item The homotopy type of $\Ind_r(G)$ for $G$ a fully whiskered graph is computed in \cref{theorem:fully whiskered graph}.
    \item The homotopy type of $\Ind_r(G)$ for $G$ a path graph is computed in \cref{prop:perfect morse fuction on path graph}.
    \item The homotopy type of $\Ind_r(G)$ for $G$ a cycle graph is computed in \cref{theorem:indr of cycle graph}.
    \item The homotopy type of $\Ind_r(G)$ for $G$ a perfect $m$-ary tree is computed in \cref{theorem:indrbhm general}.
\end{enumerate}
The proofs of all our theorems involve the construction of an optimal discrete Morse function.
Finally in \cref{sec_end} we outline some questions and conjectures.

\section{Preliminaries}\label{sec_prelim}
Let $G$ be a simple, undirected graph and $v \in V(G)$ be a vertex of $G$. The total number of vertices adjacent to $v$ is called {\itshape degree} of $v$, denoted deg$(v)$. If deg$(v)=1$, then $v$ is called a {\itshape leaf} vertex. A graph $H$ with $V(H) \subseteq V(G)$ and $E(H) \subseteq E(G)$ is called a \textit{subgraph} of the graph $G$. For a nonempty subset $U$ of $V(G)$, the induced subgraph $G[U]$, is the subgraph of $G$ with vertices $V(G[U]) = U$ and $E(G[U]) = \{(a, b) \in E(G) : a, b \in U\}$. In this article, $G[V(G)\setminus A]$ will be denoted by $G-A$ for $A\subsetneq V(G)$.

\begin{definition}
An {\itshape (abstract) simplicial complex} $\K$ on a finite (vertex) set $V(\K)$ is a collection of subsets such that 
\begin{itemize}
    \item[$(i)$] $\emptyset \in \K$, and 
    \item[$(ii)$] if $\sigma \in \K$ and $\tau \subseteq \sigma$, then $\tau \in \K$.
\end{itemize}
\end{definition}
The elements  of $\K$ are called {\itshape simplices} of $\K$.  
If $\sigma \in \K$ and $|\sigma |=k+1$, then $\sigma$ is said to be {\itshape $k$-dimensional} (here, $|\sigma|$ denotes the cardinality of $\sigma$ as a set). Further, if $\sigma \in \K$ and $\tau \subseteq \sigma$ then $\tau$ is called a {\itshape face} of $\sigma$ and if $\tau \neq \sigma$ then $\tau$ is called a {\itshape proper face} of $\sigma$. 
The elements of the underlying set $V(\K)$ are called either $0$-dimensional simplices or vertices of $\K$. 
A {\itshape subcomplex} of a simplicial complex $\K$ is a simplicial complex whose simplices are contained in $\K$. For $s\geq 0$, the {\itshape $k$-skeleton} of a simplicial complex $\K$, denoted $\K^{(s)}$, is the collection of all those simplices of $\K$ whose dimension is at most $s$.
 In this article, we do not distinguish between an abstract simplicial complex and its geometric realization. Therefore, a simplicial complex will be considered as a topological space, whenever needed.

Let $\mathbb{S}^r$ denote the $r$-sphere and let $\ast$ denote the join of two topological spaces. 
The following results will be used repeatedly in this article.

\begin{lemma}[{\cite[Lemma 2.5]{BW95}}]\label{lemma:join of spheres} Suppose that $\K_1$ and $\K_2$ are two finite simplicial complexes.
\begin{enumerate}
    \item If $\K_1$ and $\K_2$ both have the homotopy type of a wedge of spheres, then so does $\K_1\ast \K_2$.
    \item $\Big{(}\bigvee\limits_{i} \mathbb{S}^{a_i}\Big{)} \ast \Big{(}\bigvee\limits_{j} \mathbb{S}^{b_j}\Big{)} \simeq \bigvee\limits_{i,j} \mathbb{S}^{a_i+b_j+1}$
\end{enumerate}
\end{lemma}

We now discuss some tools needed from discrete Morse theory. The classical reference for this is \cite{Forman98}. However, here we closely follow \cite{Koz07} for notations and definitions.

\begin{definition}[{\cite[Definition 11.1]{Koz07}}]
A {\itshape partial matching} on a poset $P$ is a subset $\mathcal{M} \subseteq P \times P$ such that
\begin{itemize}
\item[(i)] $(a,b) \in \mathcal{M}$ implies $ a \prec b;$ {\itshape i.e.}, $a < b$ and no $c$ satisfies $ a < c < b $, and
\item[(ii)] each $ a \in P $ belong to at most one element in $\mathcal{M}$.
\end{itemize}
\end{definition}

Note that $\mathcal{M}$ is a  partial matching on a poset $P$ if and only if there exists $\mathcal{A} \subset P$ and an injective map $\mu: \mathcal{A}
 \rightarrow P\setminus \mathcal{A}$ such that $\mu(a)\succ a$ for all $a \in \mathcal{A}$. 

An {\itshape acyclic matching} is a partial matching  $\mathcal{M}$ on the poset $P$ such that there does not exist a cycle
 \begin{align*}
 		\mu(a_1)  \succ a_1 \prec \mu( a_2) \succ a_2  \prec \mu( a_3) \succ a_3 \dots   \mu(a_t) \succ a_t  \prec \mu(a_1), t\geq 2.
 \end{align*}

For an acyclic partial matching on $P$, those elements of $P$ which do not belong to the matching are called 
{\itshape critical }.

The main result of discrete Morse theory is the following.

\begin{theorem}[{\cite[Theorem 11.13]{Koz07}}]\label{acyc3}
Let $\K$ be a simplicial complex and $\mathcal{M}$ be an acyclic matching on the face poset of $\K$. Let $c_i$ denote the number of critical $i$-dimensional cells of $\K$ with respect to the matching $\mathcal{M}$. Then $\K$ is homotopy equivalent to a cell complex $\K_c$ with $c_i$ cells of dimension $i$ for each $i \geq 0$, plus a single $0$-dimensional cell in the case where the empty set is also paired in the matching.
\end{theorem}

Following can be inferred from Theorem \ref{acyc3}.

\begin{corollary}\label{acyc4}
If an acyclic matching has critical cells only in a fixed dimension $i$, then $\K$ is homotopy equivalent to a wedge of $i$-dimensional spheres.
\end{corollary} 

\begin{corollary}\label{acyc5}
If the critical cells of an acyclic matching on $\K$ form a subcomplex $\K^\prime$ of $\K$, then $\K$
simplicially collapses to $\K^\prime$, implying that $\K^\prime$ is homotopy equivalent to $\K$.
\end{corollary}

In this article, by matching on a simplicial complex $\K$, we will mean that the matching is on the face poset of $\K$.

\begin{definition}
An acyclic matching on a simplicial complex is called perfect if all the critical cells are homology cells. 
\end{definition}
 
% \begin{equation*}
% M(x) =\{(\sigma , \ \sigma \cup  \{x\}) : x\notin   \sigma,~ \sigma \cup  \{x\} \in \K\}.
% \end{equation*}

\begin{definition}
For a simplicial complex $\K$ and a vertex $x$ the matching defined as \[\{(\sigma , \ \sigma \cup  \{x\}) : x\notin   \sigma,~ \sigma \cup  \{x\} \in \K\}\] 
is called an {\itshape element matching} on $\K$ using $x$.
\end{definition}

The following result tells us that any sequence of element matchings is always acyclic.

\begin{lemma}[{\cite[Lemma 3.2]{NA18}}]\label{lemma:element matching}
For any $x\in V(\K),$ the element matching defined using $x$ is an acyclic matching on $\K$ and a perfect acyclic matching on the subcomplex $\{\sigma \in \K : \sigma \cup  \{x\} \in \K\}$.
\end{lemma}

%To obtain an acyclic matching on a simplicial complex $\K$, the next result tells us that one can define a sequence of element matchings on $\K$ using its vertices. 

\begin{proposition}[{\cite[Proposition 3.1]{SSA19}}]\label{proposition:sequence of element matching}
Let $\K$ be a simplicial complex and $\{x_1,x_2,\dots,x_n\} $ be a subset of vertex set of $\K$. 
Moreover, let $\Delta_0 = \K$ and for $i \in \{1,\dots,n\}$ define
\begin{equation}\label{eq:elementmatchingsequence}
    \begin{split}
        M(x_i) & = \{ (\sigma , \ \sigma \cup  \{x_i\}) : x_i\notin   \sigma, \text{ and } \sigma, \sigma \cup  \{x_i\} \in \Delta_{i-1} \},\\
        N(x_i) & = \{\sigma \in \Delta_{i-1} : \sigma \in \eta \mathrm{~for~some~} \eta \in M(x_i)\}, \text{ and}\\
        \Delta_{i} & =\Delta_{i-1} \setminus N(x_i).
    \end{split}
\end{equation}
Then, $\bigsqcup\limits_{i=1}^{n} M(x_i)$ is an acyclic matching on $\K$.
\end{proposition}

Another useful way to construct an acyclic matching on a poset $P$ is to first map $P$ to some other poset $Q$, then construct acyclic matchings on the fibers of this map and patch these acyclic matchings together to form an acyclic matching for the whole poset. 

\begin{theorem}[Patchwork theorem {\cite[Theorem 11.10]{Koz07}}]\label{theorem:patchwork theorem}
If $ \varphi \ : \ P \rightarrow Q$ is an order-preserving map and for each $q \in Q$, the subposet $\varphi^{-1}(q)$ carries an acyclic matching $M_q$, then ${\underset{q \ \in \ Q }{\bigsqcup}} M_q $ is an acyclic matching on P.
\end{theorem}

The following result is a special case of \cref{theorem:patchwork theorem}.
\begin{theorem}[{\cite[Lemma 4.3]{Jon08}}]\label{theorem:jacobs result of pataching two matchings}
Let $\K_0$ and $\K_1$ be disjoint families of subsets of a finite set such that $\tau \nsubseteq \sigma$ if $\sigma \in \K_0$ and $\tau \in \K_1$. If $M_i$ is an acyclic matching on $\K_i$ for $i=0,1$ then $M_0 \cup M_1$ is an acyclic matching on $\K_0 \cup \K_1$. 
\end{theorem}

\section{Basic results for higher independence complex}\label{sec_basic}

We begin this section by exploring some basic results related to the main object of this article, {\itshape i.e.}, higher independence complexes. 
Henceforth, unless otherwise mentioned, $r\geq 1$ is a natural number and $[n]$ will denote the set $\{1,\dots,n\}$.

\begin{definition}
 Let $G$ be a graph and $A\subseteq V(G)$. Then $A$ is called {\itshape $r$-independent} if connected components of $G[A]$ have cardinality at most $r$. 
\end{definition}

\begin{definition}
Let $G$ be a graph and $r \in \mathbb{N}$. The {\itshape $r$-independence complex} of $G$, denoted $\Ind_r(G)$ has vertex set $V(G)$ and its simplices are all $r$-independent subsets of $V(G)$.
\end{definition}

\begin{example}
\normalfont
\cref{fig:example of ind complex} shows a graph $G$, its $1$-independence complex and the $2$-independence complex. The $1$-independence complex of $G$ consists of $2$ maximal simplices, namely $\{v_2,v_3,v_4\}$ and $\{v_1\}$. 
The complex $\Ind_2(G)$ consists of $4$ maximal simplices, namely $\{v_1,v_2\}, \{v_1,v_3\}, \{v_1, v_4\}$ and $\{v_2,v_3,v_4\}$.
\end{example}

\begin{figure}[H]
	\begin{subfigure}[]{0.30 \textwidth}
		\centering
		\begin{tikzpicture}
 [scale=0.4, vertices/.style={draw, fill=black, circle, inner sep=1.0pt}]
        \node[vertices, label=below:{$v_1$}] (v1) at (0,0)  {};
		\node[vertices, label=below:{$v_2$}] (v2) at (4.2,0)  {};
		\node[vertices, label=above:{$v_3$}] (l11) at (-2.5,3)  {};
		\node[vertices, label=above:{$v_4$}] (l12) at (2,3.2)  {};
		
\foreach \to/\from in {v1/v2}
%\draw [-] (\to)--(\from);
\path (v1) edge node[pos=0.5,below] {} (v2);
\path (v1) edge node[pos=0.5,left] {} (l11);
\path (v1) edge node[pos=0.5,left] {} (l12);
\end{tikzpicture}\caption{$G$}\label{fig:G221}
	\end{subfigure}
	\begin{subfigure}[]{0.30 \textwidth}
		\centering
	\begin{tikzpicture}
 [scale=0.21, vertices/.style={draw, fill=black, circle, inner sep=1.0pt}]
\node[vertices, label=left:{$v_2$}] (a) at (0,0) {};
\node[vertices, label=right:{$v_3$}] (b) at (6,4) {};
\node[vertices, label=left:{$v_1$}] (c) at (7,12) {};
\node[vertices, label=right:{$v_4$}] (f) at (10,0) {};

\foreach \to/\from in {a/b,a/f,b/f}
\draw [-] (\to)--(\from);
%\filldraw[fill=gray!60, draw=black] (2,3)--(9, 6.5)--(15,3)--cycle;
\filldraw[fill=gray!60, draw=black] (0,0)--(6,4)--(10,0)--cycle;
%\draw [dashed] (a)--(b);
\end{tikzpicture}\caption{$\Ind_1(G)$}
	\end{subfigure}
	\begin{subfigure}[]{0.30 \textwidth}
		\centering
	\begin{tikzpicture}
 [scale=0.21, vertices/.style={draw, fill=black, circle, inner sep=1.0pt}]
\node[vertices, label=left:{$v_2$}] (a) at (0,0) {};
\node[vertices, label=right:{$v_3$}] (b) at (6,4) {};
\node[vertices, label=left:{$v_1$}] (c) at (7,12) {};
\node[vertices, label=right:{$v_4$}] (f) at (10,0) {};

\foreach \to/\from in {a/b,a/f,a/c,b/c,f/c,b/f}
\draw [-] (\to)--(\from);
%\filldraw[fill=gray!60, draw=black] (2,3)--(9, 6.5)--(15,3)--cycle;
\filldraw[fill=gray!60, draw=black] (0,0)--(6,4)--(10,0)--cycle;
%\draw [dashed] (a)--(b);
\end{tikzpicture}\caption{$\Ind_2(G)$}
	\end{subfigure}
	\caption{} \label{fig:example of ind complex}
\end{figure}

 The following are some easy observations that follow from the definition.

\begin{observation}\label{observation:basic properties}
\begin{enumerate}
    \item[$(i)$] For any graph $G$, $\Ind_r(G)$ is $(r-2)$-connected. Moreover, if $r\geq |V(G)|$ then $\Ind_r(G)\simeq \{\mathrm{point}\}$.
    \item[$(ii)$] If G is connected graph and $|V(G)|=r+1$, then $\Ind_r(G) \simeq \mathbb{S}^{r-1}$.
    \item[$(iii)$] Let $K_n$ be the complete graph on $n$ vertices, then $\Ind_r(K_n)$ is equal to the $(r-1)^{th}$ skeleton of an $(n-1)$-simplex, denoted $\Delta^{n-1}$, \emph{i.e.,} 
$$\Ind_r(K_n)=(\Delta^{n-1})^{(r-1)}.$$
    \item[$(iv)$] If $G$ and $H$ are two disjoint graphs, then $$\Ind_r(G \sqcup H)\simeq  \Ind_r(G) \ast \Ind_r(H).$$
    \item[$(v)$] If $G$ has a non-empty connected component of cardinality at most $r$, then $\Ind_r(G)$ is contractible.
\end{enumerate}
\end{observation}

In \cref{observation:basic properties}$(iii)$, we saw that $\Ind_r(K_n)$ is homotopy equivalent to a wedge of spheres of dimension $r-1$. 
This suggests that one should expect a similar result for complete $s$-partite graphs for $s\geq 2$. 
Recall that a {\itshape complete $s$-partite graph} is a graph in which vertex set can be decomposed into $s$ disjoint sets $V_1,V_2,\dots,V_s$ such that no two vertices within the same set $V_i$ are adjacent and if $v\in V_i$ and $w\in V_j$ for $i \neq j$ then $v$ is adjacent to $w$. 

\begin{theorem}\label{theorem:complete r partite graph}
Let $s\geq 2$ and $r \geq 1$. Given $m_1,m_2,\dots, m_s \geq 1$, the homotopy type of $r^{\text{th}}$ independence complex of the complete $s$-partite graph $K_{m_1,\dots,m_s}$ is given as follows,
$$ \Ind_r(K_{m_1,\dots,m_s}) \simeq \bigvee\limits_{t} \mathbb{S}^{r-1},$$
 where $t= \displaystyle \binom{M-1}{r} -  \sum_{i=1}^{s} \binom{m_i-1}{r}$ and $M := \sum_{i=1}^{s}m_i$ 
\end{theorem}
\begin{proof}
For simplicity of notations, we denote $K_{m_1,\dots,m_s}$ by $G$ in this proof. Let $V_1,V_2,\dots,V_s$ be the partition of vertices of $G$ and $V_i= \{v_i^1,\dots,v_i^{m_i}\}$ for $i \in [s]$. We now define a sequence of element matching on $\Delta_0:=\Ind_r(G)$ using vertices $\{v_1^1,v_2^1,\dots,v_s^1\}$ as in \Cref{eq:elementmatchingsequence}.
From \cref{proposition:sequence of element matching}, we get that $M= \bigsqcup\limits_{i=1}^{s}M(v_i^1)$ is an acyclic matching with $\Delta_s$ as the set of the critical cells.

\begin{claim}\label{claim:critical for complete r partite}
The set of critical cells after $s^{\text{th}}$ element matching is given as follows: 
\begin{equation*}
\begin{split}
\Delta_s =  &  \{\sigma \in \Ind_r(G): |\sigma|=r, ~v_i^1\notin \sigma ~\forall i \in[s] \mathrm{~and~} \sigma \nsubseteq V_i \mathrm{~for~any~} i \in [s]\} \bigsqcup \\
 & \{\sigma \in \Ind_r(G): |\sigma|=r, ~v_1^1 \notin \sigma \mathrm{~and~}~v_i^1 \in \sigma \mathrm{~for~some~} i \in \{2,\dots,s\} \}.
\end{split}
\end{equation*}
\end{claim}

\begin{proof}[Proof of \cref{claim:critical for complete r partite}]
Clearly, if $|\sigma|=r, ~v_i^1\notin \sigma ~\forall i \in[s] \mathrm{~and~} \sigma \nsubseteq V_i \mathrm{~for~any~} i \in [s]$ then $G[\sigma \cup v_i^1]$ is a connected graph of cardinality $r+1$ implying that $\sigma \notin N(v_i^1)$ for all $i \in [s]$. Therefore, $\{\sigma \in \Ind_r(G): |\sigma|=r, ~v_i^1\notin \sigma ~\forall i \in[s] \mathrm{~and~} \sigma \nsubseteq V_i \mathrm{~for~any~} i \in [s]\} \subseteq \Delta_s$. Now, let $|\sigma|=r$ and $v_i^1\in \sigma$ for some $i \in \{2,\dots,s\}$. For $i \in \{2,\dots,s\}$, if $v_i^1 \in \sigma$ then $\sigma\setminus v_i^1 \in N(v_1^1)$ implies that $\sigma \notin N(v_i^1)$. If $v_j^1 \notin \sigma$, then $|\sigma|=r$ and $v_i^1 \in \sigma$ for some $i\neq j$ implies that $G[\sigma \cup v_j^1]$ is connected subgraph of cardinality $r+1$, hence $\sigma \notin N(v_j^1)$. Thus $\{\sigma \in \Ind_r(G): |\sigma|=r, ~v_1^1 \notin \sigma \mathrm{~and~} v_i^1 \in \sigma \mathrm{~for~some~} i \in \{2,\dots,s\} \} \subseteq \Delta_s$.

Now consider $\sigma \in \Delta_s$. If $\sigma \subseteq V_1$ or $|\sigma|<r$ or $v_1^1 \in \sigma$, then $\sigma \in N(v_1^1)$. If $\sigma \subseteq V_i$ for some $i\in [s]$ and $v_i^1 \notin \sigma$
then $\sigma \cup v_i^1 \in \Ind_r(G)$. 
Hence $\sigma \in N(v_i^1)$, however, this is a contradiction to the fact that $\sigma \in \Delta_s$. 
Thus, either $\sigma \nsubseteq V_i$ for any $i\in [s]$ or if $\sigma \subseteq V_i$ for some $i \in \{2,\dots,s\}$ then $v_i^1 \in \sigma$. Now, let $|\sigma|>r$. $\sigma \in \Ind_r(G)$ implies that $\sigma \subseteq V_i$ for some $i \in [s]$ but then $\sigma \in N(v_i^1)$. Therefore, $|\sigma|=r$. This completes the proof of \cref{claim:critical for complete r partite}.
\end{proof}
Using \cref{claim:critical for complete r partite}, we get that $M$ is an acyclic matching on $\Ind_r(G)$ with exactly $|\Delta_s|$ critical cells of dimension $(r-1)$. Therefore, \cref{acyc4} implies that $\Ind_r(G)$ is homotopy equivalent to a wedge of $|\Delta_s|$ spheres of dimension $r-1$. We now compute the cardinality of the set $\Delta_s$. Using \cref{claim:critical for complete r partite}, we get
\begin{equation*}
    \begin{split}
        |\Delta_s| & = \binom{\sum\limits_{i=1}^{s}m_i -s} {r} -  \sum\limits_{i=1}^{s} \binom{m_i-1}{r}+\sum\limits_{j=2}^{s}\binom{\sum\limits_{i=1}^{s}m_i - j}{r-1}\\
        & = {\binom{\sum\limits_{i=1}^{s}m_i-1}{r}}-  \sum\limits_{i=1}^{s} \binom{m_i-1}{r}
    \end{split}
\end{equation*}
This completes the proof of \cref{theorem:complete r partite graph}.
\end{proof}

We now show that adding a whisker (a leaf vertex) at each vertex of $G$ simplifies the homotopy type of higher independence complex. By adding a whisker at vertex $v$ of $G$, we mean a new vertex is attached to $v$ (the induced subgraph $K_2$ is called {\itshape whisker}). We show that the higher independence complex of fully whiskered graphs is homotopy equivalent to a wedge of equi-dimensional spheres. 

\begin{definition}
Given a graph $G$, a {\itshape fully whiskered graph} of $G$, denoted $W(G)$, is a graph in which a whisker is added to each vertex of $G$.
\end{definition}

\begin{figure}[H]
	\begin{subfigure}[]{0.45 \textwidth}
		\centering
		\begin{tikzpicture}
		[scale=0.5, vertices/.style={circle, draw = black!100, fill= gray!0, inner sep=0.5pt}] 
		\node (a) at (0,0)  {$a_1$};
		\node (b) at (4,0) {$a_2$};
		\node (c) at (8,0)  {$a_3$};
		
		\foreach \from/\to in {a/b, b/c}
		\draw (\from) -- (\to);
		\end{tikzpicture}
		\caption{$P_{3}$}
	\end{subfigure}
	\begin{subfigure}[]{0.45 \textwidth}
		\centering
		\begin{tikzpicture}
		[scale=0.5, vertices/.style={circle, draw = black!100, fill= gray!0, inner sep=0.5pt}]
		%[scale=0.55, vertices/.style={draw, fill=black, circle, inner sep=1.5pt}]
	    \node (a) at (0,0)  {$a_1$};
		\node (b) at (4,0) {$a_2$};
		\node (c) at (8,0)  {$a_3$};
		\node (d) at (0,3)  {$b_{1}$};
		\node (f) at (4,3)  {$b_{2}$};
		\node (g) at (8,3)  {$b_{3}$};
		
		\foreach \from/\to in {a/b, a/d, b/c, b/f, c/g}
		\draw (\from) -- (\to);
		\end{tikzpicture}
		\caption{$W(P_3)$}
	\end{subfigure}
	\caption{} \label{fig:fully whiskered graph}
\end{figure}

\begin{theorem}\label{theorem:fully whiskered graph}
Let $G$ be a connected graph and $V(G)=\{a_1,a_2,\dots,a_n\}$ be the set of vertices of $G$. The homotopy type of $\Ind_r(W(G))$ is given by the following formula:

$$\Ind_r(W(G)) \simeq 
\begin{cases}\bigvee\limits_{ \binom{n-1}{r-n}} \mathbb{S}^{r-1}, & \mathrm{if\ }  n \leq r \leq 2n-1,\\
\{\mathrm{point}\}, & \mathrm{otherwise}.
\end{cases}$$
\end{theorem}

\begin{proof}
Let $\{b_{1},b_{2},\dots,b_{n}\}$ denote the set of leaves of graph $W(G)$ such that $b_i$ is adjacent to $a_i$ for each $i \in [n]$. 
We define a sequence of element matching on $\Delta_0=\Ind_r(W(G))$ using vertices $\{b_{1},b_{2},\dots,b_{n}\}$ as in \Cref{eq:elementmatchingsequence}.
From \cref{proposition:sequence of element matching}, we get that $M= \bigsqcup\limits_{i=1}^{s}M(b_i)$ is an acyclic matching on $\Ind_r(W(G))$ with $\Delta_n$ as the set of the critical cells.

\begin{claim}\label{claim:fully whiskered graph}
If $\sigma \in \Ind_r(W(G))$ and $V(G) \nsubseteq \sigma$ then $\sigma \notin \Delta_n$, {\itshape i.e.} $\sigma$ is not a critical cell.
\end{claim}
Let $p=\text{min}\{i : a_i \notin \sigma\}$. It is easy to observe that, if $\sigma \in \Delta_{p-1}$ then $\sigma$ belongs to $N(b_{p})$, which implies that $\sigma \notin \Delta_{n}$. This prove \cref{claim:fully whiskered graph}.

We are now ready to prove \Cref{theorem:fully whiskered graph}. First, let $r< n$. Since $G$ is connected, if $\sigma \in \Ind_r(W(G))$ then $V(G) \nsubseteq\sigma$. Hence, result follows from \cref{claim:fully whiskered graph} and \cref{acyc4}. 

Now, assume that $r \geq n$. 
From the definition of $\Ind_r(G)$, it is easy to see that if $\sigma \in \Ind_r(G)$ and the cardinality of $\sigma$ is less than $r$ then $\sigma \in N(b_{1})$. Thus, if $\sigma \in \Delta_{n}$ then the cardinality of $\sigma$ is at least $r$ and $b_{1} \notin \sigma$.  
Using \cref{claim:fully whiskered graph}, we see that if $\sigma \in \Delta_n$ then $V(G) \subseteq \sigma$.  Further, if $\sigma \in \Ind_r(G)$ and $V(G) \subseteq \sigma$ then $\sigma \notin N(b_{i})$ for any $i\in [n]$.
Hence $\sigma \in \Delta_n$ iff $V(G)\subseteq \sigma$, $b_{1} \notin \sigma$ and $| \sigma| \geq r$. 
Moreover, $V(G)\subseteq \sigma$ implies that $G[\sigma]$ is always connected. 
Therefore, the cardinality of $\sigma$ is exactly $r$. Combining all these arguments, we see that $\Delta_n$ is a set of $\displaystyle \binom{n-1}{r-n}$ cells of dimension $r-1$. 
The result follows from \cref{acyc4}.
\end{proof}

We now show that, for a graph $G$, adding more whiskers at non-leaf vertices of $W(G)$ does not affect the connectivity of the higher independence complex. In particular, we give closed form formula for the homotopy type of $r$-independence complexes of these new graphs.

\begin{theorem}\label{theorem:leaf at every vertex general graph}
Let $G$ be a connected graph and $W=\{a_1,a_2,\dots,a_n\}$ be the set of all non-leaf vertices of $G$. For $i \in \{1,\dots,n\}$, let $l_i$ denote the number of leaves adjacent to vertex $a_i$. If $l_i > 0$ for all $i \in \{1,\dots,n\}$, then the homotopy type of $\Ind_r(G)$ is given as follows.

$$\Ind_r(G) \simeq 
\begin{cases}\bigvee\limits_t \mathbb{S}^{r-1}, & \mathrm{if\ }  r \geq n,\\
\{\mathrm{point}\}, & \mathrm{otherwise},
\end{cases}$$

where $t =\displaystyle \binom{\sum_{i=1}^n l_i -1}{r-n}$.
\end{theorem}
\begin{proof} Arguments in this proof are similar to that of the proof of \cref{theorem:fully whiskered graph}. For $i \in [n]$, let $\{b_{i,1},b_{i,2},\dots,b_{i,l_i}\}$ denote the set of leaves adjacent to $a_i$. Let $\Delta_0=\Ind_r(G)$. We define a sequence of element matching on $\Delta_0$ using leaf vertices $\{b_{1,1},b_{2,1},\dots,b_{n,1}\}$ as in \Cref{eq:elementmatchingsequence}. 

\begin{claim}\label{claim: leaf at all vertex}
If $\sigma \in \Ind_r(G)$ and $W \nsubseteq \sigma$ then $\sigma \notin \Delta_n$, {\itshape i.e.} $\sigma$ is not a critical cell.
\end{claim}
Let $p=\text{min}\{i : a_i \notin \sigma\}$. It is easy to observe that, if $\sigma \in \Delta_{p-1}$ then $\sigma$ belongs to $N(b_{p,1})$, which implies that $\sigma \notin \Delta_{n}$. This prove \cref{claim: leaf at all vertex}.

Let $r< n$. Since $G$ is connected and $W$ is collection of all non-leaf vertices, $G[W]$ is connected subgraph of cardinality $n$. Therefore, if $\sigma \in \Ind_r(G)$ then $W \nsubseteq\sigma$.  Hence, result follows from \cref{claim: leaf at all vertex} and \cref{acyc4}. 

Now, assume that $r \geq n$. 
From the definition of $\Ind_r(G)$, it is easy to see that if $\sigma \in \Ind_r(G)$ and the cardinality of $\sigma$ is less than $r$ then $\sigma \in N(b_{1,1})$. 
Hence, for a simplex $\sigma \in \Delta_{n}$ the cardinality of $\sigma$ is at least $r$ and $b_{1,1} \notin \sigma$.  
Using \cref{claim: leaf at all vertex}, we see that $\sigma \in \Delta_n$ implies $W\subseteq \sigma$.  
Further, if $\sigma \in \Ind_r(G)$ and $W\subseteq \sigma$ then $\sigma \notin N(b_{i,1})$ for any $i\in [n]$. 
One concludes that $\sigma \in \Delta_n$ iff $W\subseteq \sigma$, $b_{1,1} \notin \sigma$ and $| \sigma| \geq r$. 
Moreover, $W\subseteq \sigma$ implies that $G[\sigma]$ is always connected. 
Therefore, the cardinality of the simplex $\sigma$ is exactly $r$. Combining all these arguments, we see that $\Delta_n$ is a set of $\displaystyle \binom{\sum_{i=1}^n l_i -1}{r-n}$ cells of dimension $r-1$. Thus the result follows from \cref{acyc4}.
\end{proof}

For $n \geq 1$, a {\itshape path graph} of length $n$, denoted $P_n$, is a graph with vertex set $V(P_n) = \{1, \ldots, n\}$ and  edge set $E(P_n) = \{(i,i+1) \ |  \ 1 \leq i \leq n-1\}$. For $n \geq 3$, a {\itshape cycle graph}, denoted $C_n$, is a graph with vertex set $V(C_n) = \{1, \ldots, n\}$ and  edge set $E(C_n) = \{(i,i+1) \ |  \ 1 \leq i \leq n-1\} \cup \{(1,n)\}$.
We can now compute $r$-independence complexes of almost all caterpillar graphs. A {\itshape caterpillar} graph is a path graph with some whiskers on vertices. 
\begin{definition}
Let $G$ be a graph with $V(G)=\{a_1,\dots,a_n\}$ and $L=\{ l_1,\dots,l_n\}$ be a set of $n$ non-negative integers. Define a graph $G^L$ with the following data:
\begin{equation*}
    \begin{split}
        V(G^L)& =V(G) \sqcup \bigsqcup\limits_{l_i>0}\{b_{i,1},\dots,b_{i,l_i}\}\\
        E(G^L)& = E(G) \sqcup \bigsqcup\limits_{l_i >0} \{(a_i,b_{i,j}): 1 \leq j \leq l_i \}\\
    \end{split}
\end{equation*}
\end{definition}

See \cref{fig:path and cycle star} for examples. Clearly, $P_n^L$ is a caterpillar graph.

\begin{figure}[H]
	\begin{subfigure}[]{0.45 \textwidth}
		\centering
		\begin{tikzpicture}
		[scale=0.5, vertices/.style={circle, draw = black!100, fill= gray!0, inner sep=0.5pt}] 
		\node (a) at (0,0)  {$a_1$};
		\node (b) at (4,0) {$a_2$};
		\node (c) at (8,0)  {$a_3$};
		\node (d) at (-1,3)  {$b_{1,1}$};
		\node (e) at (1,3)  {$b_{1,2}$};
		\node (f) at (4,3)  {$b_{2,1}$};
		\node (g) at (8,3)  {$b_{3,1}$};
		
		\foreach \from/\to in {a/b, a/d, a/e, b/c, b/f, c/g}
		\draw (\from) -- (\to);
		\end{tikzpicture}
		\caption{$P_{3}^{(2,1,1)}$}
	\end{subfigure}
	\begin{subfigure}[]{0.45 \textwidth}
		\centering
		\begin{tikzpicture}
		[scale=0.5, vertices/.style={circle, draw = black!100, fill= gray!0, inner sep=0.5pt}]
		%[scale=0.55, vertices/.style={draw, fill=black, circle, inner sep=1.5pt}]
	    \node (a) at (0,0)  {$a_1$};
		\node (b) at (3,3) {$a_2$};
		\node (c) at (6,0)  {$a_3$};
		\node (d) at (-2,-1)  {$b_{1,1}$};
		\node (e) at (-2,1)  {$b_{1,2}$};
		\node (f) at (3,5)  {$b_{2,1}$};
		\node (g) at (8,3)  {$b_{3,1}$};
		
		\foreach \from/\to in {a/b, a/c, a/d, a/e, b/c, b/f, c/g}
		\draw (\from) -- (\to);
		\end{tikzpicture}
		\caption{$C_3^{(2,1,1)}$}
	\end{subfigure}
	\caption{} \label{fig:path and cycle star}
\end{figure}

\begin{corollary}
Given $L=(l_1,l_2,\dots,l_n)$ with $l_i > 0$ for every $i \in \{1,\dots,n\}$. Then,

$$\Ind_r(P_n^{L}) \simeq \Ind_r(C_n^{L}) \simeq
\begin{cases}\bigvee\limits_{\binom{\sum_{i=1}^n l_i -1}{r-n}} \mathbb{S}^{r-1}, & \mathrm{if\ }  r \geq n,\\
\{\mathrm{point}\}, & \mathrm{otherwise}.
\end{cases}$$
\end{corollary}

\section{Higher Independence Complexes of cycle graphs}\label{sec_cycle}
Kozlov, in \cite{Koz07}, computed the homotopy type of $1$-independence complex of cycle graphs using discrete Morse theory. He proved the following result:

\begin{proposition}[{\cite[Proposition 11.17]{Koz07}}]
For any $n\geq 3$, we have
\[ \Ind_1(C_n) \simeq
\begin{cases}
\mathbb{S}^{k-1}\bigvee \mathbb{S}^{k-1}, & \mathrm{if~} n=3k,\\
\mathbb{S}^{k-1},  & \mathrm{if~} n=3k \pm 1.
\end{cases}
\]
\end{proposition}
In this section, we generalize this result and compute the homotopy type of $\Ind_r(C_n)$ for any $n \geq 3$ and $r\geq 1$. In particular, we define a perfect acyclic matching on $\Ind_{d-2}(C_n)$. We will use the following result, proved by Paolini and Salvetti in \cite{PS18}. 

\begin{theorem}[{\cite[Proposition 3.7]{PS18}}]\label{higher_ind_path} For $d\geq 3$, we have
\[\Ind_{d-2}(P_n) \simeq
\begin{cases} 
     \mathbb{S}^{dk-2k-1}, & \mathrm{if\ } n = dk \mathrm{\  or\ } n=dk-1;\\
     \{\mathrm{point}\}, & \mathrm{otherwise}.
   \end{cases}
\]
\end{theorem}

To make our computations of $\Ind_{d-2}(C_n)$ easier, we first improve the acyclic matching defined by Paolini and Salvetti on $\Ind_r(P_n)$, and get a perfect acyclic matching on $\Ind_{d-2}(P_n)$.

\begin{proposition}\label{prop:perfect morse fuction on path graph}
There exists a perfect acyclic matching on $\Ind_{d-2}(P_{n})$. 
In particular, if $n=dk$ or $dk-1$ then the only critical cell is given by $ \bigsqcup\limits_{i=0}^{k-1}\{di+2,\dots,di+d-1\}$. 
Moreover, for other values of $n$ there are no critical cells.
\end{proposition}
\begin{proof}
Let $n=dk-t$ for some $t \in \{0,1,\dots,d-1\}$, let $\Delta=\{\sigma \in \Ind_{d-2}(P_{n}) : \sigma \cap\{d,2d,\dots,dk\} \neq \emptyset\}$ and let $\Delta_0=\Ind_{d-2}(P_{n}) \setminus \Delta$. In {\cite[Proposition 3.7]{PS18}}, Paolini and Salevtti constructed an acyclic matching $\mathcal{M}$ on $\Ind_{d-2}(P_{n})$ with $\Delta_0$ as the set of critical cells. Here, we construct an acyclic matching on $\Delta_0$ using vertices $\{1,d+1,2d+1,\dots,d(k-1)+1\}$ as in \Cref{eq:elementmatchingsequence}.
From \cref{proposition:sequence of element matching}, $\mathcal{M}^\prime = \bigsqcup\limits_{i=0}^{k-1}M(di+1)$ is an acyclic matching on $\Delta_0$ with $\Delta_k$ as the set of critical cells. Clearly, if $n=dk$ or $dk-1$ then $\Delta_k= \{\sigma\}$, where $\sigma=\bigsqcup\limits_{i=0}^{k-1}\{di+2,\dots,di+d-1\}$. Further, if  $n\neq dk, dk-1$ then $N(d(k-1)+1)=\Delta_{k-1}$. Using \cref{theorem:jacobs result of pataching two matchings}, we get that $\mathcal{M} \sqcup \mathcal{M}^\prime$ is an acyclic matching on $\Ind_{d-2}(P_{n})$ with $\Delta_k$ as set of critical cells. This completes the proof of \cref{prop:perfect morse fuction on path graph}.
\end{proof}

Following are some immediate corollaries of \cref{prop:perfect morse fuction on path graph}.

\begin{corollary}\label{cor:disjoint_union_of_paths_2}
Let $d\geq 3$ and $G$ be disjoint union of $m$ path graphs of lengths $d$ or $d-1$. Then there exists an acyclic matching on $\Ind_{d-2}(G)$ with exactly one critical cell of dimension $0$ and one of dimension $(d-3)m+m-1=dm-2m-1$.
\end{corollary}

\begin{corollary}\label{cor:disjoint_union_of_paths_3}
Let $d\geq 3$ and $G$ be disjoint union of $m$ path graphs. If any connected component of $G$ has length less than $d-1$ or greater than $d$ and less than $2d-2$, then there exists an acyclic matching on $\Ind_{d-2}(G)$ with no critical cell.
\end{corollary}

From \cref{observation:basic properties}$(i)$ and $(ii)$, we get that $\Ind_{d-2}(C_n) \simeq \{\text{point}\}$ for all $n\leq d-2$ and $\Ind_{d-2}(C_{d-1})$ $ \simeq \mathbb{S}^{d-3}$. We now determine the homotopy type of $\Ind_{d-2}(C_{n})$ for $n\geq d$. The idea of this proof is to define an acyclic matching of subsets of the face poset of $\Ind_r(C_n)$ and then use \cref{theorem:patchwork theorem}.

\begin{theorem}\label{theorem:indr of cycle graph}
For $n\geq d \geq 3$, we have
\[\Ind_{d-2}(C_n) \cong
\begin{cases} 
     \displaystyle \bigvee_{d-1}\mathbb{S}^{dk-2k-1}, & \mathrm{if\ } n =dk;\\
     \mathbb{S}^{dk-2k-1}, & \mathrm{if\ } n = dk+ 1;\\
     \mathbb{S}^{dk-2k}, & \mathrm{if\ } n = dk+2;\\
     \vdots & \vdots\\
     \mathbb{S}^{dk-2k+d-3}, & \mathrm{if\ } n = dk+(d-1).
   \end{cases}
\]
\end{theorem}

\begin{proof}
In this proof, we assume that the vertices of $C_n$ are labeled $1,2,\dots,n$ in anti-clockwise direction. 
Let $k$ denote the maximal integer such that $dk \leq n$. Furthermore, let $E$ be a chain with $k+1$ elements labeled as follows:
$$e_d > e_{2d} > \cdots  > e_{dk} > e_{r}.$$ 

Let $\cF(\K)$ denote the set of faces of a simplicial complex $\K$. We define a map 
\begin{equation}\label{eq:definition of map phi cycle graph} 
\phi: \cF(\Ind_{d-2}(C_n)) \rightarrow E
\end{equation} 
by the following rule. The simplices that contain the vertex labeled $d$ get mapped to $e_d$; the simplices that do not contain the vertex labeled $d$, but contain the vertex labeled $2d$ get mapped to $e_{2d}$; the simplices that do not contain the vertices labeled $d$ and $2d$, but contain the vertex labeled $3d$ get mapped to $e_{3d}$; and so on. 
Finally, the simplices that does not contain any of the vertices labeled $d, 2d,\dots , dk$ all get mapped to $e_r$.

Clearly, the map $\phi$ is order-preserving, since if one takes a larger simplex, it will have more vertices, and the only way its image may change is to go up when a new element from the set $\{d, 2d, \dots , dk\}$ is added and is smaller than the previously smallest one.

Let us now define acyclic matchings on the preimages of elements of $E$ under the map $\phi$. We split our argument into cases.

\textbf{Case 1:} We first consider the preimages $\phi^{-1}(e_{2d})$ through $\phi^{-1}(e_{dk})$. Let $t$ be an integer such that $2 \leq t \leq k$. The preimage $\phi^{-1}(e_{dt})$ consists of all simplices $\sigma$ such that $d, 2d, \dots , d(t-1) \not \in \sigma$, while $dt \in \sigma$. Since $\sigma \in \Ind_{d-2}(C_n)$, $\{dt-1,dt-2,\dots,dt-(d-2)\} \nsubseteq \sigma$. This means that the pairing $\sigma \leftrightarrow \sigma \cup \{dt-(d-1)\}$ provides a well-defined matching, which is acyclic from \cref{lemma:element matching}.

\textbf{Case 2:} Next, we consider the preimage $\phi^{-1}(e_d)$. For $\sigma \in \Ind_{d-2}(C_n)$, let conn$_d(\sigma)$ is the number of vertices of connected component of $C_n[\sigma]$ containing vertex labeled $d$. We define a map $\psi: \phi^{-1}(e_d) \rightarrow \{c_1 < c_2 < \dots < c_{d-2}\}$
\[\psi(\sigma)=c_{\mathrm{conn_d}(\sigma) }\]
Clearly, $\psi$ is a poset map and for $i \in \{1,\dots,d-2\}$, if $\sigma \in \psi^{-1}(c_i)$ then cardinality of $\sigma$ is at least $i$.

For $t \geq 1$, let $P_t^{\{i+1,\dots,i+t\}}$ denote the path graph of length $t$ whose vertices are labeled as $i+1,i+2,\dots,i+t$ (see \cref{fig:example of path graph}). 

\begin{figure}[H]
		\centering
		\begin{tikzpicture}
  [scale=0.4, vertices/.style={draw, fill=black, circle, inner sep=1.0pt}]
        \node[vertices, label=above:{$i+1$}] (v1) at (0,0)  {};
		\node[vertices, label=above:{$i+2$}] (v2) at (6,0)  {};
		\node[vertices, label=above:{$i+t-2$}] (vn2) at (10,0)  {};
		\node[vertices, label=above:{$i+t-1$}] (vn1) at (16,0)  {};
		\node[vertices, label=above:{$i+t$}] (vn) at (22,0)  {};
		\node[vertices,inner sep=0.3pt] (d1) at (7.5,0)  {};
		\node[vertices,inner sep=0.3pt] (d2) at (8,0)  {};
		\node[vertices,inner sep=0.3pt] (d3) at (8.5,0)  {};
		
\foreach \to/\from in {v1/v2,vn2/vn1,vn1/vn}
\draw [-] (\to)--(\from);
\end{tikzpicture}
	\caption{$P_t^{\{i+1,\dots,i+t\}}$} \label{fig:example of path graph}
\end{figure}
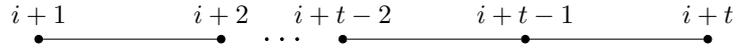

We now define a matching on $\phi^{-1}(e_d)$ in $d-2$ steps as follows.

{\bfseries Step $1$:} 
 For $p\geq 1$, it is clear that the $p$-simplices of $\psi^{-1}(c_1)$ are in 1-1 correspondence with the $(p-1)$-simplices of $\Ind_{d-2}(P_{n-3}^{\{d+2,\dots,n,1,\dots,d-2\}})$ with one extra simplex of dimension $0$, which is $\{d\}$. Using \cref{prop:perfect morse fuction on path graph}, let $M_0$ be a perfect matching on $\Ind_{d-2}(P_{n-3}^{\{d+2,\dots,n,1,\dots,d-2\}})$. Define a matching $M_1$ on $\psi^{-1}(c_1)$ as follows: $(\sigma,\tau) \in M_0$ iff $(\sigma\cup d,\tau\cup d) \in M_1$. Therefore, we get the following.

\begin{itemize}
    \item Matching $M_1$ is an acyclic matching on $\psi^{-1}(c_1)$ with the following property. If $n-3=dk-1$ or $n-3=dk$, \emph{i.e.,} $n=dk+2$ or $dk+3$, then there is only one critical cell of dimension $dk-2k$ and that is 
    \begin{equation}
    \begin{split}
    \{d\} \sqcup \bigsqcup\limits_{i=1}^{k-1}\{di+3,\dots,d(i+1)\} \sqcup \{1 ,\dots,d-2\}, & \mathrm{~~if~} n=dk+2, \\
    \{d\} \sqcup \bigsqcup\limits_{i=1}^{k-1}\{di+3,\dots,d(i+1)\} \sqcup \{ n,1 ,\dots,d-3\}, & \mathrm{~~if~} n=dk+3.
    \end{split}
    \end{equation} 
    Otherwise, there is no critical cell.  % can we elaborate on this?? like initially there will be two critical cell from thm 3.1 and its pairs with empty cell(which actually becomes singleton
\end{itemize}

\textbf{Step $2$:} Observe that, in $C_n$, there are exactly two connected subgraphs of cardinality two containing vertex $d$, which are $C_n[\{d-1,d\}]=P_2^{\{d-1,d\}}$ and $C_n[\{d,d+1\}]=P_2^{\{d,d+1\}}$. Thus, elements of $\psi^{-1}(c_2)$ can be partitioned into two smaller disjoint subsets $\Delta_{\{d-1,d\}}$ and $\Delta_{\{d,d+1\}}$. Here, $\Delta_{\{d-1,d\}}$ is collection of all those simplices $\sigma \in \psi^{-1}(c_2)$ such that $\{d-1,d\}$ is the connected component of $C_n[\sigma]$. Similarly, $\Delta_{\{d,d+1\}}$ is collection of all those simplices $\sigma \in \psi^{-1}(c_2)$ such that $\{d,d+1\}$ is the connected component of $C_n[\sigma]$. Clearly, $\psi^{-1}(c_2)= \Delta_{\{d-1,d\}} \cup \Delta_{\{d,d+1\}}$ and $\Delta_{\{d-1,d\}} \cap \Delta_{\{d,d+1\}} = \emptyset$. Now, the idea is to define acyclic matching on $\Delta_{\{d-1,d\}}$, $\Delta_{\{d,d+1\}}$ and merge them together to get an acyclic matching on $\psi^{-1}(c_2)$.

\begin{enumerate}
    \item  Observe that, for $p\geq 2$, the $p$-simplices of $\Delta_{\{d-1,d\}}$ are in 1-1 correspondence with the $(p-2)$-simplixes of $\Ind_{d-2}(P_{n-4}^{\{d+2,\dots,n,1,\dots,d-3\}})$ with one extra simplex of dimension $1$, which is $\{d-1,d\}$. Using \cref{prop:perfect morse fuction on path graph}, let $M$ be a perfect matching on $\Ind_{d-2}(P_{n-4}^{\{d+2,\dots,n,1,\dots,d-3\}})$. Define a matching $M_2^1$ on $\Delta_{\{d-1,d\}}$ as follows: $(\sigma,\tau) \in M$ iff $(\sigma\cup \{d-1,d\},\tau\cup \{d-1,d\}) \in M_2^1$. Therefore, we get the following.
    
    Matching $M_2^1$ is an acyclic matching on $\Delta_{\{d-1,d\}}$ with the following property. If $n-4=dk-1$ or $dk$, \emph{i.e.,} $n=dk+3$ or $dk+4$, then there is only one critical cell of dimension $dk-2k+1$ and that is 
    \begin{equation}
    \begin{split}
    \{d-1,d\} \sqcup \bigsqcup\limits_{i=1}^{k-1}\{di+3,\dots,d(i+1)\} \sqcup \{n,1 ,\dots,d-3\}, & \mathrm{~~if~} n=dk+3, \\
    \{d-1,d\} \sqcup \bigsqcup\limits_{i=1}^{k-1}\{di+3,\dots,d(i+1)\} \sqcup \{n-1,n,1 ,\dots,d-4\}, & \mathrm{~~if~} n=dk+4.\\
    \end{split}
    \end{equation} 
    Otherwise, there is no critical cell.  
    
    \item Similar to the case of $\Delta_{\{d-1,d\}}$ and using the matching of $\Ind_{d-2}(P_{n-4}^{\{d+3,\dots,n,1,\dots,d-2\}})$, we get an acyclic matching, say $M_2^2$ on $\Delta_{\{d,d+1\}}$ with the following property.
    
    If $n-4=dk-1$ or $dk$, \emph{i.e.,} $n=dk+3$ or $dk+4$, then there is only one critical cell of dimension $dk-2k+1$ and that is 
    \begin{equation}
    \begin{split}
    \{d,d+1\} \sqcup \bigsqcup\limits_{i=1}^{k-1}\{di+4,\dots,d(i+1)+1\} \sqcup \{1 ,\dots,d-2\}, & \mathrm{~~if~} n=dk+3, \\
    \{d,d+1\} \sqcup \bigsqcup\limits_{i=1}^{k-1}\{di+4,\dots,d(i+1)+1\} \sqcup \{n,1 ,\dots,d-3\}, & \mathrm{~~if~} n=dk+4.\\
    \end{split}
    \end{equation}
    Otherwise, there is no critical cell.
\end{enumerate}
 Since $\psi^{-1}(c_2) = \Delta_{\{d-1,d\}} \sqcup \Delta_{\{d,d+1\}}$, $M_2=M_2^1\sqcup M_2^2$ (defined above) is an acyclic matching on $\psi^{-1}(c_2)$ with exactly two critical cells of dimension $dk-2k+1$ whenever $n=dk+3$ or $dk+4$ and with no critical cell otherwise. 
 
 We now define a matching on $\psi^{-1}(c_{d-2})$. Idea here is similar to that of step $2$.
 
{\bfseries Step $d-2$:} Observe that, in $C_n$, there are exactly $d-2$ connected subgraphs  of cardinality $d-2$ containing vertex $d$, and these subgraphs are path graphs of length $d-2$, {\itshape i.e.,} one of the element of the following set:  $\mathcal{L}= \big{\{}L^{\{3,4,\dots,d-1,d\}}_{d-2}, L^{\{4,5,\dots,d-1,d,d+1\}}_{d-2}, \dots, L^{\{d,d+1,\dots,2d-4,2d-3\}}_{d-2} \big{\}}$. Thus, elements of $\psi^{-1}(c_{d-2})$ can be partitioned into $d-2$ smaller disjoint subsets $\Delta_{L}$ for each $L\in \mathcal{L}$. Here, $\Delta_{L}$ is collection of all those simplices $\sigma \in \psi^{-1}(c_{d-2})$ such that $L$ is the connected component of $C_n[\sigma]$. Clearly, $\psi^{-1}(c_{d-2})= \bigsqcup\limits_{L\in \mathcal{L}}\Delta_{L}$. Now, the idea is to define acyclic matchings on $\Delta_{L}$ for each $L \in \mathcal{L}$ and merge them together to get an acyclic matching on $\psi^{-1}(c_{d-2})$.

\begin{enumerate}
    \item  Observe that, for $p\geq d-2$, the $p$-simplices of $\Delta_{L^{\{3,4,\dots,d-1,d\}}_{d-2}}$ are in 1-1 correspondence with the $(p-(d-2))$-simplices of $\Ind_{d-2}(P_{n-d}^{\{d+2,\dots,n,1\}})$ with one extra simplex of dimension $d-3$, which is $\{3,4,\dots,d-1,d\}$. Using \cref{prop:perfect morse fuction on path graph}, let $M$ be a perfect matching on $\Ind_{d-2}(P_{n-d}^{\{d+2,\dots,n,1\}})$. Define a matching $M_{d-2}^3$ on $\Delta_{L^{\{3,4,\dots,d-1,d\}}_{d-2}}$ as follows: $(\sigma,\tau) \in M$ iff $(\sigma\cup \{3,4,\dots,d-1,d\},\tau\cup \{3,4,\dots,d-1,d\}) \in M_{d-2}^3$. Therefore, we get the following.
    
    Matching $M_{d-2}^3$ is an acyclic matching on $\Delta_{L^{\{3,4,\dots,d-1,d\}}_{d-2}}$ with the following property. If $n-d=dk-1$ or $dk$, \emph{i.e.,} $n=d(k+1)-1$ or $d(k+1)$, then there is only one critical cell of dimension $dk-2k-1+d-2=d(k+1)-2(k+1)-1$ and that is 
    \begin{equation}
    \begin{split}
    \{3,4,\dots,d-1,d\} \sqcup \bigsqcup\limits_{i=1}^{k-1}\{di+3,\dots,d(i+1)\} \sqcup \{dk+3 ,\dots,n,1\}, & \mathrm{~~if~} n=d(k+1)-1, \\
    \{3,4,\dots,d-1,d\} \sqcup \bigsqcup\limits_{i=1}^{k-1}\{di+3,\dots,d(i+1)\} \sqcup \{dk+3 ,\dots,n\}, & \mathrm{~~if~} n=d(k+1).\\
    \end{split}
    \end{equation}
    Otherwise, there is no critical cell.
    
    \item  We now define a matching on $\Delta_{L^{\{t,t+1,\dots,d+t-3\}}_{d-2}}$ for each $t \in \{4,5,\dots,d\}$. Similar to the case of $\Delta_{L^{\{3,4,\dots,d-1,d\}}_{d-2}}$, we define an acyclic matching on $\Delta_{L^{\{t,t+1,\dots,d+t-3\}}_{d-2}}$, say $M_{d-2}^t$ using the perfect matching defined on $\Ind_{d-2}(P_{n-d}^{\{d+2,\dots,n,1\}})$. We thus get the following.
    
    If $n-d=dk-1$ or $dk$, \emph{i.e.,} $n=d(k+1)-1$ or $d(k+1)$, then there is only one critical cell of dimension $dk-2k-1+d-2=d(k+1)-2(k+1)-1$ and that is 
    \begin{equation}
    \begin{split}
     & \{t,t+1,\dots,d+t-3\} \sqcup \bigsqcup\limits_{i=1}^{k-1}\{di+t,\dots,d(i+1)+t-3\} \sqcup \{dk+t ,\dots,n,1,\dots,t-2\}, \\ & \mathrm{~~if~} n=d(k+1)-1 \mathrm{~~and~~}\\
    & \{t,t+1,\dots,d+t-3\} \sqcup \bigsqcup\limits_{i=1}^{k-1}\{di+t,\dots,d(i+1)+t-3\} \sqcup \{dk+t ,\dots,n,1,\dots,t-3\}, \\ & \mathrm{~~if~} n=d(k+1).
    \end{split}
    \end{equation}
    Otherwise, there is no critical cell.
\end{enumerate}
Since $\psi^{-1}(c_{d-2})= \bigsqcup\limits_{L\in \mathcal{L}}\Delta_{L}$, $M_{d-2}=\bigsqcup\limits_{t=3}^{d}M_{d-2}^t$ (defined in step $d-2$) is an acyclic matching on $\psi^{-1}(c_{d-2})$ with exactly $d-2$ critical cells of dimension $d(k+1)-2(k+1)-1$ whenever $n=d(k+1)-1$ or $d(k+1)$ and with no critical cell otherwise.

Using \cref{theorem:patchwork theorem}, we observe that $M= \bigsqcup\limits_{i=1}^{d-2} M_i$ is an acyclic matching on $\phi^{-1}(e_d)$ with: 
\begin{itemize}
    \item no critical cell if $n=dk+1$,
    \item exactly $1$ critical cell of dimension $dk-2k$ if $n=dk+2$
    \item exactly $t-2$ critical cells of dimension $dk-2k+t-3$ and $t-1$ critical cells of dimension $dk-2k+t-2$, if $n=dk+t$ for some $t \in \{3, \dots,d-1\}$
    \item exactly $d-2$ critical cells of dimension $d(k+1)-2(k+1)-1$ if $n=d(k+1)$.
\end{itemize}

We now define another matching on the set of critical cells corresponding to matching $M$ on $\phi^{-1}(e_d)$. The idea is the following. If $n=dk+3$, then observe from step $1$ and step $2$ that if $\gamma$ is a critical cell of dimension $dk-2k$ then $\gamma \cup \{d-1\}$ is also a critical cell of dimension $dk-2k+1$. Thus, match $\gamma$ with $\gamma \cup \{d-1\}$. Now, let $n=dk+t$ for some $t \in \{4,\dots,d-1\}$. From step $t-2$ and step $t-1$ we see that, if in step $t-2$, $ \gamma= \{d-i,\dots,d,\dots,d+t-i-3\}\cup \{\beta\}$ is a critical cell of dimension $dk-2k+t-3$ then in step $t-1$, $\{d-i-1,d-i,\dots,d,\dots,d+t-i-3\}\cup \{\beta\}$ is critical cell of dimension $dk-2k+t-2$. Here, we match $\gamma$ with $\gamma \cup \{d-i-1\}$. Let the matching defined above is $M^\prime$.

\begin{claim}\label{claim:for matching on inverse image of d} Let $M$ and $M^\prime$ be matchings on $\phi^{-1}(e_d)$ as defined above. Then, $\mathcal{M}=M \sqcup M^\prime$ is an acyclic matching on $\phi^{-1}(e_d)$ with 
\begin{itemize}
    \item no critical cell if $n=dk+1$,
    \item exactly $1$ critical cell of dimension $dk-2(k+1)+t$ if $n=dk+t$ for some $t \in \{2,\dots,d-1\}$,
    \item exactly $d-2$ critical cells of dimension $d(k+1)-2(k+1)-1$ if $n=d(k+1)$,
\end{itemize}
\end{claim}
\begin{proof}[Proof of \cref{claim:for matching on inverse image of d}]
Let $\Delta_0=\{\sigma \in \phi^{-1}(e_d): \sigma \in \eta \mathrm{~for ~ some~} \eta \in M \}$ and $\Delta_1= \phi^{-1}(e_d) \setminus \Delta_0$. Since $M$ and $M^\prime$ are union of a sequence of elementary matchings on $\Delta_0$ and $\Delta_1$ respectively, $M$ and $M^\prime$ are acyclic matching from \cref{proposition:sequence of element matching}. 

Further, it is clear from the description of the critical cells given in steps $1$ through $(d-2)$ that if $\tau \in \Delta_1$ and $\sigma \in \Delta_0$ then $\tau \nsubseteq \sigma$. Thus, using \cref{theorem:jacobs result of pataching two matchings}, we get that $\mathcal{M}$ is an acyclic matching on $\phi^{-1}(e_d)$. Calculation of number of critical cells corresponding to matching $\mathcal{M}$ is straight forward once we fix an $n$.
\end{proof}

\textbf{Case 3:} In cases $1$ and $2$, we defined acyclic matchings on $\phi^{-1}(e_{di})$ for $i \in \{1,\dots,k\}$. Here, we consider the preimage $\phi^{-1}(e_r)$ and define a matching $\mathcal{M}^\prime$ on it.

\begin{itemize}
    \item If $n=dk$, then $\phi^{-1}(e_r)$ is isomorphic to $\Ind_{d-2}(G)$, where $G$ is isomorphic to the union $k$ disjoint copies of path graphs of length $d-1$. From \cref{cor:disjoint_union_of_paths_2}, there exists an acyclic matching on the face poset of $\Ind_{d-2}(G)$ with exactly one critical cell of dimension $dk-2k-1$.
    \item If $n=dk+1$, then $\phi^{-1}(e_r)$ is isomorphic to $\Ind_{d-2}(G_1)$, where $G_1$ is isomorphic to the union $k-1$ disjoint copies of $P_{d-1}$ and one copy of $P_d$. Again from \cref{cor:disjoint_union_of_paths_2}, there exists an acyclic matching on the face poset of $\Ind_{d-2}(G_1)$ with exactly one critical cell of dimension $dk-2k-1$.
    \item If $n\neq dk,dk+1$ then one connected component of $C_n\setminus \{d,2d,\dots,dk\}$ will be a path graph of cardinality either less than $d-1$ or greater than $d$ and less than $2d-2$. In both the cases, using \cref{cor:disjoint_union_of_paths_3} there exists a matching on $\phi^{-1}(c_r)$ with no critical cell.
    \end{itemize}
From \Cref{eq:definition of map phi cycle graph}, \cref{theorem:patchwork theorem}, case $(1)$, \cref{claim:for matching on inverse image of d} and case $3$, we get that $\mathcal{M}\cup \mathcal{M}^\prime$ is an acyclic matching on $\cF(\Ind_{d-2}(C_n))$ with
\begin{itemize}
    \item exactly $d-1$ critical cells of dimension $(dk-2k-1)$ if $n=dk$,
    \item exactly one critical cell of dimension $(dk-2k+t-2)$ if $n=dk+t$ for some $t \in \{1,\dots,d-1\}$.
\end{itemize}
Hence, \cref{theorem:indr of cycle graph} follows from \cref{acyc4}.
\end{proof}

\section{The case of perfect \texorpdfstring{$m$}{m}-ary trees}\label{sec_trees}

For fixed $m\geq 2$, an $m$-ary tree is a rooted tree in which each node has no more than $m$ children. A full $m$-ary tree is an $m$-ary tree where within each level every node has either $0$ or $m$ children. A perfect $m$-ary tree is a full $m$-ary tree in which all leaf nodes are at the same depth (the depth of a node is the number of edges from the node to the tree's root node). 

\begin{figure}[H]
\begin{subfigure}[]{0.25 \textwidth}
\vspace{1.35cm}
    \Tree[.$a_{0,1}$ [.$a_{1,1}$ [.$a_{2,1}$ ]
                [.$a_{2,2}$  ]]
               [.$a_{1,2}$ $a_{2,3}$ 
                [.$a_{2,4}$ ]] ]
    \caption{$B_2^2$}  
    \label{graphb2}
\end{subfigure}
 \begin{subfigure}[]{0.7 \textwidth}
    \Tree[ .$a_{0,1}$ [.$a_{1,1}$ [.$a_{2,1}$ ] [.$a_{2,2}$ ]
                [.$a_{2,3}$  ]]
                [.$a_{1,2}$ [.$a_{2,4}$ ] [.$a_{2,5}$ ]
                [.$a_{2,6}$  ]]
                [.$a_{1,3}$ [.$a_{2,7}$ ] [.$a_{2,8}$ ]
                [.$a_{2,9}$  ]]]
    \caption{$B_2^3$}  
    \label{graphb4}
\end{subfigure}
\begin{subfigure}[]{0.10 \textwidth}
\end{subfigure}
\caption{}
\label{ptreegraph}
\end{figure}

Following are some known facts about the perfect $m$-ary tree of height $h$, denoted $B_h^m$ (see \cref{ptreegraph} for example).
\begin{enumerate}
\item $B_h^m$ has $\sum\limits_{i=0}^{h} m^{i}= \frac{m^{h+1}-1}{m-1}$ nodes.
\item For $0 \leq t\leq h$, the number of nodes of depth $t$ in $B_h^m$ is $m^{t}$.
\item $B_h^m$ has $m^{h}$ leaf nodes.
\end{enumerate}

Let us first fix some notations.

\begin{itemize} 
    
    \item For $d \in \{0,1,\dots,h\}$, let $V_d(B_h^m)$ denote the set of vertices of $B_h^m$ of depth $d$. 
    
    \item Let the vertices of $B_h^m$ of depth $d$ be labelled by $a_{d,1},a_{d,2},\dots, a_{d,m^{d}}$ from left to right (see \cref{ptreegraph}).
    
    \item The following ordering of the vertices of $B_h^m$ will be used in the proofs of this section. Given $a_{p,q},a_{p^\prime,q^\prime} \in V(B_h^m)$, we say that $a_{p,q}< a_{p^\prime,q^\prime}$ whenever $q < q^\prime$ and if $q=q^\prime$ then $p< p^\prime$. For example, in $B_2^3$, $a_{0,1}<a_{1,1}<a_{2,1}<a_{2,4}$. 
    
     \item For $\sigma \in \Delta$, denote $\sigma \cup \{v\}$ by $\sigma \cup v$.
     
     \item For simplicity of notations, $B_h^2$ will be denoted by $B_h$.
\end{itemize}

We first give some examples to explain our method for computing the homotopy type of higher independence complexes of $B_h$. 

\begin{example}\label{exampleind4b2}
\normalfont 
Here we compute the homotopy type of $\Ind_4(B_2)$. 
Define an element matching on $\Delta_0=\Ind_4(B_2)$ using the vertex $a_{2,1}$ as in \Cref{eq:elementmatchingsequence}.

Observe that, if $\sigma \in \Delta_1$ then $\sigma\cup a_{2,1} \notin \Ind_4(B_2)$. By definition of $\Ind_r(G)$, we observe that either $\{a_{1,1},a_{0,1},a_{1,2},a_{2,2}\} \subseteq \sigma$ or $ \{a_{1,1},a_{0,1},a_{1,2},a_{2,3}\} \subseteq \sigma$ or $\{a_{1,1},a_{0,1},a_{1,2},a_{2,4}\} \subseteq \sigma$. Since $\{a_{1,1},a_{0,1},a_{1,2},a_{2,2}\} , \{a_{1,1},a_{0,1},a_{1,2},a_{2,3}\} , \{a_{1,1},a_{0,1},a_{1,2},a_{2,4}\}$ are maximal simplices of $\Ind_4(B_2)$, these are the only unmatched cells {\itshape i.e.}, 
$$\Delta_1=\big{\{}\{a_{1,1},a_{0,1},a_{1,2},a_{2,2}\}, \{a_{1,1},a_{0,1},a_{1,2}, a_{2,3}\} , \{a_{1,1},a_{0,1},a_{1,2},a_{2,4}\}\big{\}}.$$ Therefore, \cref{acyc4} implies that $\Ind_4(B_2)\simeq \bigvee\limits_{3} \mathbb{S}^3$.
\end{example}

\begin{example}\label{exampleind4b3}
\normalfont 
Using the homotopy type of $\Ind_4(B_2)$, we compute the homotopy type of $\Ind_4(B_3)$. Here, we show that $\Ind_4(B_3)\simeq \Ind_4(B_3 - \{a_{0,1}\})$. It is easy to see that $B_3 - \{a_{0,1}\} \cong B_2 \sqcup B_2$ (see \Cref{fig:b3andb3minus}). Thus, \cref{observation:basic properties}$(iv)$ implies that $\Ind_4(B_3)\simeq \Ind_4(B_2) \ast \Ind_4(B_2) \simeq \bigvee\limits_{9} \mathbb{S}^7$.

\begin{figure}[H]
	\begin{subfigure}[]{0.5 \textwidth}
		\centering
		\begin{tikzpicture}
 [scale=0.40, vertices/.style={inner sep=0.3pt}]
		\node[vertices] (a01) at (0,0)  {$a_{0,1}$};
		\node[vertices] (a11) at (-4,-3)  {$a_{1,1}$};
		\node[vertices] (a12) at (4,-3)  {$a_{1,2}$};
		\node[vertices] (a21) at (-6,-6)  {$a_{2,1}$};
		\node[vertices] (a22) at (-2,-6)  {$a_{2,2}$};
		\node[vertices] (a23) at (2,-6)  {$a_{2,3}$};
		\node[vertices] (a24) at (6,-6)  {$a_{2,4}$};
		\node[vertices] (a31) at (-7,-9)  {$a_{3,1}$};
		\node[vertices] (a32) at (-5,-9)  {$a_{3,2}$};
		\node[vertices] (a33) at (-3,-9)  {$a_{3,3}$};
		\node[vertices] (a34) at (-1,-9)  {$a_{3,4}$};
		\node[vertices] (a38) at (7,-9)  {$a_{3,8}$};
		\node[vertices] (a37) at (5,-9)  {$a_{3,7}$};
		\node[vertices] (a36) at (3,-9)  {$a_{3,6}$};
		\node[vertices] (a35) at (1,-9)  {$a_{3,5}$};
		
\foreach \to/\from in {a01/a11,a01/a12,a11/a21,a11/a22,a12/a23,a12/a24,a21/a31,a21/a32,a22/a33,a22/a34,a23/a35,a23/a36,a24/a37,a24/a38}
\draw [-] (\to)--(\from);
\end{tikzpicture}\caption{$B_3$}\label{fig:G243}
	\end{subfigure}
	\begin{subfigure}[]{0.5 \textwidth}
		\centering
		\vspace{1cm}
	\begin{tikzpicture}
  [scale=0.40, vertices/.style={inner sep=0.3pt}]
		\node[vertices] (a11) at (-4,-3)  {$a_{1,1}$};
		\node[vertices] (a12) at (4,-3)  {$a_{1,2}$};
		\node[vertices] (a21) at (-6,-6)  {$a_{2,1}$};
		\node[vertices] (a22) at (-2,-6)  {$a_{2,2}$};
		\node[vertices] (a23) at (2,-6)  {$a_{2,3}$};
		\node[vertices] (a24) at (6,-6)  {$a_{2,4}$};
		\node[vertices] (a31) at (-7,-9)  {$a_{3,1}$};
		\node[vertices] (a32) at (-5,-9)  {$a_{3,2}$};
		\node[vertices] (a33) at (-3,-9)  {$a_{3,3}$};
		\node[vertices] (a34) at (-1,-9)  {$a_{3,4}$};
		\node[vertices] (a38) at (7,-9)  {$a_{3,8}$};
		\node[vertices] (a37) at (5,-9)  {$a_{3,7}$};
		\node[vertices] (a36) at (3,-9)  {$a_{3,6}$};
		\node[vertices] (a35) at (1,-9)  {$a_{3,5}$};
		
\foreach \to/\from in {a11/a21,a11/a22,a12/a23,a12/a24,a21/a31,a21/a32,a22/a33,a22/a34,a23/a35,a23/a36,a24/a37,a24/a38}
\draw [-] (\to)--(\from);
\end{tikzpicture}\caption{$B_3-\{a_{0,1}\}$}\label{b3minus}
	\end{subfigure}
	\caption{} \label{fig:b3andb3minus}
\end{figure}

We now prove that $\Ind_4(B_3)\simeq  \Ind_4(B_3 - \{a_{0,1}\})$. Let $R(a_{0,1})=\{\sigma \in \Ind_4(B_3) : a_{0,1} \in \sigma\}$. Clearly, $\Ind_4(B_3)\setminus R(a_{0,1}) = \Ind_4(B_3 - \{a_{0,1}\})$. From \cref{acyc5}, it is enough to define a perfect matching on $R(a_{0,1})$. We do so by defining a sequence of elementary matching on $\Delta_0=\Ind_4(B_3)$ using vertices $\{a_{3,1},a_{3,3},a_{3,5},a_{3,7}\}$ as in \Cref{eq:elementmatchingsequence}. We now compute the set of critical cells, {\itshape i.e}, $\Delta_4$.

\begin{claim}\label{clainforind4b3}
$\Delta_4=\Ind_4(B_3)\setminus R(a_{0,1})$.
\end{claim}
Since $N(a_{3,2i-1})\subseteq R(a_{0,1})$ for all $i \in \{1,2,3,4\}$, $\Ind_4(B_3)\setminus R(a_{0,1})\subseteq \Delta_4$. To show the other way inclusion, it is enough to show that if $\sigma \in \Ind_4(B_3)$ and $a_{0,1}\in \sigma$ then $\sigma \in N(a_{3,2i-1})$ for some $i \in \{1,2,3,4\}$.

Let $\sigma \in \Ind_4(B_3)$ and $a_{0,1}\in \sigma$. 
Since $a_{0,1}\in \sigma$, it follows from the definition of $\Ind_r(G)$ that $\{a_{1,1},a_{1,2},a_{2,1},a_{2,2}\} \nsubseteq \sigma$.  If $a_{3,1}\in \sigma$ or $|\sigma| < 4$ or  $\{a_{1,1},a_{2,1}\} \nsubseteq \sigma$, then $\sigma \in N(a_{3,1})$ Thus, assume that $a_{3,1}\notin \sigma$ and $|\sigma| \geq 4$ and  $\{a_{1,1},a_{2,1}\} \subset \sigma$. Now, if $a_{2,2} \notin \sigma$ then $\sigma \in N(a_{3,3})$ and if $a_{2,2} \in \sigma$ then $a_{1,2} \notin \sigma$ which implies that $\sigma \in N(a_{3,5})$. This completes the proof of \cref{clainforind4b3}.
\end{example}

Henceforth, $m \geq 3$ will be a fixed integer.

\begin{lemma}\label{lemma:indrbhm limited}
Let $r\geq \frac{m^h-1}{m-1}$. Then the homotopy type of $r$-independence complex of the graph $B_h^m$ is given as follows,

$$ \Ind_r(B_h^m)\simeq 
\begin{cases}
\bigvee\limits_{\binom{m^h-1}{s}} \mathbb{S}^{r-1}, & \mathrm{if } ~r = \frac{m^h-1}{m-1}+s \mathrm{ ~for ~some ~} s \in \{0,1,\dots, m^h-1\},  \\
\{\mathrm{point} \}, & \mathrm{if } ~r \geq \frac{m^{h+1}-1}{m-1}.
\end{cases}$$
\end{lemma}
\begin{proof}
The idea of the proof here is similar to that of \cref{exampleind4b2}. If $r\geq \frac{m^{h+1}-1}{m-1}$, then \cref{observation:basic properties}$(i)$ implies the result. Let $r = \frac{m^h-1}{m-1}+s$ for some fixed $s \in \{0,1,\dots, m^h-1\}$ and $\Delta_0=\Ind_r(B_h^m)$. Define a sequence of element matching on $\Delta_0$ as in \Cref{eq:elementmatchingsequence} using the following vertices of depth $h$: $\{a_{h,1},a_{h,m+1},\dots,a_{h,m(m^{h-1}-1)+1}\}.$ 

We now show that the set of critical cells $\Delta_{m^{h-1}}$ is a set of $\binom{m^h-1}{s}$ cells of dimension $r-1$.

\begin{claim}\label{claim:indrbhm limited}
\begin{enumerate}
    \item If $\sigma \in \Delta_{m^{h-1}}$, then $\bigsqcup\limits_{j=0}^{h-1} V_j(B_h^m) \subseteq \sigma$.
    \item If $\sigma \in \Delta_{m^{h-1}}$, then $\sigma$ is of cardinality $r$.
    \item The cardinality of the set of critical cells $\Delta_{m^{h-1}}$ is $\binom{m^h-1}{s}$.
\end{enumerate}
\end{claim}

\begin{proof}[Proof of \cref{claim:indrbhm limited}]
To the contrary of \cref{claim:indrbhm limited}$(1)$, assume that there exists $\sigma_1 \in \Delta_{m^{h-1}}$ such that $\bigsqcup\limits_{j=0}^{h-1} V_j(B_h^m) \nsubseteq \sigma$. Let $a_{i_1,j_1} \in \bigsqcup\limits_{j=0}^{h-1} V_j(B_h^m)$ be the smallest element with respect to the ordering given above such that $a_{i_1,j_1} \notin \sigma_1$. Since $a_{i_1,j_1}$ is not a leaf, let $a_{i_1,j_1}^1$ be the first child of $a_{i_1,j_1}$ and let $a_{h,m\ell-(m-1)}$ be the left most leaf of the sub-tree rooted at $a_{i_1,j_1}^1$. Since $\sigma_1 \in \Delta_{m^{h-1}}$, $\sigma \in \Delta_{\ell-1}$. Observe that the number of vertices of sub-tree rooted at $a_{i_1,j_1}^1$ is not more than $\frac{m^h-1}{m-1}$. Therefore, $\sigma_1 \in N(a_{h,\ell})$ (being the left most child of a sub-tree, $\ell$ is an odd number) contradicting the assumption that $\sigma_1 \in \Delta_{m^{h-1}}$. This proves \cref{claim:indrbhm limited}$(1)$.

We now prove the second part of the above claim. Let $\sigma \in \Delta_{m^{h-1}}$. Clearly, the cardinality of $\sigma$ is at least $r$ (because any face of $\Ind_r(B_h)$ of cardinality less that $r$ is in $N(a_{h,1})$). Using \cref{claim:indrbhm limited}$(1)$, we see that $B_h[\sigma]$ is a connected graph of cardinality equal to the cardinality of $\sigma$. Therefore, the cardinality of $\sigma$ is at most $r$. This proves \cref{claim:indrbhm limited}$(2)$.

It is clear that, if $\sigma \in \Ind_r(B_h)$ and $a_{h,1} \in \sigma$ then $\sigma \in N(a_{h,1})$ implying that $\sigma \notin \Delta_{m^{h-1}}$. Hence, using \cref{claim:indrbhm limited}$(1)$ and $(2)$, we get that the cardinality of the set $\Delta_{m^{h-1}}$ is equal to number of $s$-subsets of the set $V_h(B_h^m)\setminus\{a_{h,1}\}$.  Which is equal to $\binom{m^h-1}{s}$. This completes the proof of \cref{claim:indrbhm limited}.
\end{proof}

From \cref{claim:indrbhm limited} and \cref{proposition:sequence of element matching}, we see that the matching $\bigcup\limits_{i=1}^{m^{h-1}}M(a_{h,mi-(m-1)})$ on $\Ind_r(B_h^m)$ has $\binom{m^h-1}{s}$ critical cells of fixed dimension $r-1$. Therefore, \cref{lemma:indrbhm limited} follows from \cref{acyc4}.
\end{proof}

We are now ready to present the main result of this section.

\begin{theorem}\label{theorem:indrbhm general}
For a fixed $t \geq 1$, let $r= \big{(}\sum\limits_{i=0}^{t-1}m^i \big{)} +s =\frac{m^t-1}{m-1}+s$ for some $s \in \{0,1, \dots, m^t-1\}$. Then the homotopy type of $r$-independence complex of the graph $B_h^m$ is given as follows,

$$ \Ind_r(B_h^m)\simeq 
\begin{cases}
\bigvee\limits_{p_1} \mathbb{S}^{q_1}, & \mathrm{if } ~h = (k-1)(t+2)+t+1 \mathrm{ ~for ~some }~k \geq 1,  \\
\bigvee\limits_{p_2} \mathbb{S}^{q_2}, & \text{if } h = k(t+2)+t \text{ for some } k \geq 0,  \\
\{\mathrm{point} \}, & \mathrm{otherwise},
\end{cases}$$

where, \begin{equation*}
\begin{split}
    p_1 & = \binom{m^t-1}{s}^{m(m^0+m^{t+2}+\dots+m^{(k-1)(t+2)})} \text{ and } \\
    q_1 & = mr(m^0+m^{t+2}+\dots+m^{(k-1)(t+2)})-1,\\
    p_2 & = \binom{m^t-1}{s}^{m^0+m^{t+2}+\dots+m^{k(t+2)}}, \\
    q_2 & = r(m^0+m^{t+2}+\dots+m^{k(t+2)})-1.
\end{split}
\end{equation*}
\end{theorem}

\begin{proof}
The idea here is similar to that of \cref{exampleind4b3}. If $h \leq t$, then the result follows from \cref{lemma:indrbhm limited}. Let $h>t$. Here, we show that $\Ind_r(B_h^m)\simeq \Ind_r(G)$, where $G$ is disjoint union of perfect $m$-ary trees of height at most $t$. Recall that $V_j(B_h^m)$ denotes the set of vertices of $B_h^m$ of depth $j$. 

\begin{claim}\label{claim to reduce size of graph restricted mptree}
$\Ind_r(B_h^m) \simeq \Ind_r(B_h^m - V_{h-(t+1)}(B_h^m)))$.
\end{claim} 

\begin{proof}[Proof of \cref{claim to reduce size of graph restricted mptree}]
Let $R(V_{h-(t+1)}(B_h^m))=\{\sigma \in \Ind_r(B_h^m) : \sigma \cap V_{h-(t+1)}(B_h^m) \neq \emptyset\}$. Clearly, $\Ind_r(B_h^m)\setminus R(V_{h-(t+1)}(B_h^m)) = \Ind_r(B_h^m - V_{h-(t+1)}(B_h^m))$. To prove \cref{claim to reduce size of graph restricted mptree}, using \cref{acyc5}, it is enough to define a perfect matching on $R(V_{h-(t+1)}(B_h^m))$. We do so by defining a sequence of element matching on $\Delta_0=\Ind_r(B_h^m)$ using vertices $\{a_{h,1},a_{h,m+1},\dots,a_{h,m^h- (m-1)}\}$ as in \Cref{eq:elementmatchingsequence}. 

 We now analyze the set of critical cells $\Delta_{m^{h-1}}$ corresponding to the matching $\bigcup\limits_{i=1}^{m^{h-1}}M(a_{h,mi-(m-1)})$ and prove that $\Delta_{m^{h-1}}=\Ind_r(B_h^m)\setminus R(V_{h-(t+1)}(B_h^m))$. Which, along with \cref{acyc5}, will imply \cref{claim to reduce size of graph restricted mptree}. 
 
 Since $N(a_{h,mi-(m-1)})\subseteq R(V_{h-(t+1)}(B_h^m))$ for all $i \in \{1,2,\dots,m^{h-1}\}$, $\Ind_r(B_h^m)\setminus R(V_{h-(t+1)}(B_h^m))\subseteq \Delta_{m^{h-1}}$. To show that $\Delta_{m^{h-1}} \subseteq \Ind_r(B_h^m)\setminus R(V_{h-(t+1)}(B_h^m))$, it is enough to show that if $\sigma \in \Ind_r(B_h^m)$ and $\sigma \cap V_{h-(t+1)}(B_h^m) \neq \emptyset$ then $\sigma \in N(a_{h,mi-(m-1)})$ for some $i \in \{1,2,\dots,m^{h-1}\}$ {\itshape i.e.}, $\sigma \notin \Delta_{m^{h-1}}$. We prove this by contradiction.

Let $\sigma_1 \in \Ind_r(B_h^m)$ such that $\sigma_1 \cap V_{h-(t+1)}(B_h^m) \neq \emptyset$ and $\sigma_1 \in \Delta_{m^{h-1}}$. Without loss of generality, assume that $a_{h-(t+1),\ell}$ is the smallest vertex of $V_{h-(t+1)}(B_h^m)$ such that $a_{h-(t+1),\ell} \in \sigma_1$. Let $B(a_{h-(t+1),\ell},B_h^m)$ be the sub-tree of $B_h^m$ rooted at $a_{h-(t+1),\ell}$. Let $S$ denote the set of all non-leaf vertices of $B(a_{h-(t+1),\ell},B_h^m)$, {\itshape i.e.}, $S=\bigsqcup\limits_{j=1}^{t+1} V_{h-j}(B_h^m) \bigcap V(B(a_{h-(t+1),\ell},B_h^m))$. Clearly, $B(a_{h-(t+1),\ell},B_h^m)$ is a perfect $m$-ary tree of height $t+1$ and the cardinality of $S$ is $\frac{m^{t+1}-1}{m-1}$. Since $B_h^m[S]$ is a connected graph and $r<\frac{m^{t+1}-1}{m-1}$, $S\nsubseteq \sigma_1$. Let $a_{i_1,j_1}$ be the smallest element of $S$ such that $a_{i_1,j_1} \notin \sigma_1$. Since $a_{i_1,j_1} \in S$ and $a_{h-(t+1),\ell} \in \sigma_1$, we get that $i_1 \in \{h-t,h-t+1,\dots,h-1\}$. Let $a_{i_1+1,j_2}$ be the left most child of $a_{i_1,j_1}$ and $a_{h,m\ell_1-(m-1)}$ be the left most leaf of perfect $m$-ary sub-tree rooted at $a_{i_1+1,j_2}$. Since $\sigma_1 \in \Delta_{m^{h-1}}$, $\sigma \in \Delta_{\ell_1-1}$. Observe that the cardinality of the sub-tree rooted at $a_{i_1+1,j_2}$ is at most $\frac{m^t-1}{m-1}$. Hence, $\sigma_1 \in N(a_{h,m\ell_1-(m-1)})$ which is a contradiction. This completes the proof of \cref{claim to reduce size of graph restricted mptree}.
\end{proof}

We prove \cref{theorem:indrbhm general} using induction on $h$. 

\begin{itemize}

\item[{\bfseries  Step $1$:}]  In this step, we prove the result for $h \in \{t+1,t+2,\dots,(t+2)+t\}$.

From \cref{claim to reduce size of graph restricted mptree}, we see that $\Ind_r(B_h^m) \simeq \Ind_r(B_h^m - V_{h-(t+1)}(B_h^m)))$. Observe that $B_h^m - V_{h-(t+1)}(B_h^m)$ is disjoint union of $m(m^{h-(t+1)})$  copies of perfect $m$-ary trees of height $t$ and one perfect $m$-ary tree of height $h-(t+2)$ (here, by $B_{-1}^m$ we mean empty graph). Therefore, using \cref{observation:basic properties}$(iv)$ and \cref{lemma:indrbhm limited}, we get the following equivalence.

\begin{equation*}
\begin{split}
    \Ind_r(B_h^m) & \simeq \Ind_r(\underbrace{B_t^m \sqcup \dots \sqcup B_t^m}_{m(m^{h-(t+1)})\text{-copies}} \sqcup B_{h-(t+2)}^m) \\ 
    & \simeq \underbrace{\Ind_r(B_t^m) \ast \dots \ast \Ind_r(B_t^m)}_{m(m^{h-(t+1)})\text{-copies}} \ast \Ind_r(B_{h-(t+2)}^m).
    \end{split}
\end{equation*} 
Which implies that,
\begin{equation*}
\begin{split}
    \Ind_r(B_h^m) & \simeq \begin{cases}
            \underbrace{\Ind_r(B_t^m) \ast \dots \ast \Ind_r(B_t^m)}_{m \text{-copies}} \ast \Ind_r(B_{-1}^m), & \text{if } h=t+1, \\
            \underbrace{\Ind_r(B_t^m) \ast \dots \ast \Ind_r(B_t^m)}_{(m^{t+2})\text{-copies}} \ast \Ind_r(B_{t}^m), & \text{if } h=(t+2)+t, \\
            \underbrace{\Ind_r(B_t^m) \ast \dots \ast \Ind_r(B_t^m)}_{m(m^{h-(t+1)})\text{-copies}} \ast \{\text{point}\}, & \text{if }  t+1 < h <(t+2)+t. \\
            \end{cases} 
    \end{split}
\end{equation*}
Thus, \cref{lemma:indrbhm limited} and \cref{lemma:join of spheres} implies the result, {\itshape i.e.},
\begin{equation*}
\begin{split}
 \Ind_r(B_h^m) & \simeq \begin{cases}
            \bigvee\limits_{\binom{m^t-1}{s}^{m}} \mathbb{S}^{mr-1}, & \text{if } h=t+1, \\
             \bigvee\limits_{\binom{m^t-1}{s}^{(m^0+m^{t+2})}} \mathbb{S}^{r(m^0+m^{t+2})-1}, & \text{if } h=(t+2)+t, \\
            \{\text{point}\}, & \text{if }  t+1 < h <(t+2)+t. \\
            \end{cases}
\end{split}
\end{equation*}

\item[{\bfseries  Step $2$:}]  In this step, we prove the result for $h \in \{(t+2)+t+1,\dots,2(t+2)+t\}$.

Following similar method as in step $1$, we get the following equivalence,

$$\Ind_r(B_h^m)  \simeq \underbrace{\Ind_r(B_t^m) \ast \dots \ast \Ind_r(B_t^m)}_{m(m^{h-(t+1)})\text{-copies}} \ast \Ind_r(B_{h-(t+2)}^m)$$

Observe that $h-(t+2)$ is in $\{t+1,t+2,\dots,(t+2)+t\}$. Thus, result of Step $1$ implies the following.

\begin{equation*}
\begin{split}
    \Ind_r(B_h^m) & \simeq \begin{cases}
            \underbrace{\Ind_r(B_t^m) \ast \dots \ast \Ind_r(B_t^m)}_{m(m^0+m^{t+2})\text{-copies}} \ast \Ind_r(B_{-1}^m), & \text{if } h=(t+2)+t+1, \\
            \underbrace{\Ind_r(B_t^m) \ast \dots \ast \Ind_r(B_t^m)}_{(m^{t+2}+m^{2(t+2)})\text{-copies}} \ast \Ind_r(B_{t}^m), & \text{if } h=2(t+2)+t, \\
            \underbrace{\Ind_r(B_t^m) \ast \dots \ast \Ind_r(B_t^m)}_{m(m^{h-(t+1)})\text{-copies}} \ast \{\text{point}\}, & \text{if }  (t+2)+t+1 < h <2(t+2)+t. \\
            \end{cases} \\
     & \simeq \begin{cases}
            \underbrace{\Ind_r(B_t^m) \ast \dots \ast \Ind_r(B_t^m)}_{m(m^0+m^{t+2})\text{-copies}},  & \text{if } h=(t+2)+t+1, \\
            \underbrace{\Ind_r(B_t^m) \ast \dots \ast \Ind_r(B_t^m)}_{(m^0+m^{t+2}+m^{2(t+2)})\text{-copies}}, & \text{if } h=2(t+2)+t, \\
            \{\text{point}\}, & \text{if }  (t+2)+t+1 < h <2(t+2)+t. \\
            \end{cases} \\
    \end{split}
\end{equation*}

Using \cref{lemma:indrbhm limited} and \cref{lemma:join of spheres}, we get the result, {\itshape i.e.},
 \begin{equation*}
\begin{split}
    \Ind_r(B_h^m)  & \simeq \begin{cases}
            \bigvee\limits_{\binom{m^t-1}{s}^{m(m^0+m^{t+2})}} \mathbb{S}^{mr(m^0+m^{t+2})-1}, & \text{if } h=(t+2)+t+1, \\
             \bigvee\limits_{\binom{m^t-1}{s}^{(m^0+m^{t+2}+m^{2(t+2)})}} \mathbb{S}^{r(m^0+m^{t+2}+m^{2(t+2)})-1}, & \text{if } h=2(t+2)+t, \\
            \{\text{point}\}, & \text{if }  (t+2)+t+1 < h <2(t+2)+t. \\
            \end{cases}
    \end{split}
\end{equation*}

\item[{\bfseries  Step $k$:}]  In this step, we prove the result for $h \in \{(k-1)(t+2)+t+1,\dots,k(t+2)+t\}$ where $k\geq 3$.

The proof here is exactly similar to that of Step $2$. Therefore,
\begin{equation*}
    \Ind_r(B_h^m)  \simeq \underbrace{\Ind_r(B_t^m) \ast \dots \ast \Ind_r(B_t^m)}_{m(m^{h-(t+1)})\text{-copies}} \ast \Ind_r(B_{h-(t+2)}^m) \\
\end{equation*}
Thus, result of Step $k-1$ implies the following equivalence.
\begin{equation*}
\begin{split}
    \Ind_r(B_h^m) & \simeq \begin{cases}
            \underbrace{\Ind_r(B_t^m) \ast \dots \ast \Ind_r(B_t^m)}_{m(m^0+m^{t+2}+\dots+m^{(k-1)(t+2)})\text{-copies}} \ast \Ind_r(B_{-1}^m), & \text{if } h=(k-1)(t+2)+t+1, \\
            \underbrace{\Ind_r(B_t^m) \ast \dots \ast \Ind_r(B_t^m)}_{(m^{t+2}+m^{2(t+2)+\dots+m^{k(t+2)}})\text{-copies}} \ast \Ind_r(B_{t}^m), & \text{if } h=k(t+2)+t, \\
            \underbrace{\Ind_r(B_t^m) \ast \dots \ast \Ind_r(B_t^m)}_{m(m^{h-(t+1)})\text{-copies}} \ast \{\text{point}\}, & \text{if }  (k-1)(t+2)+t+1 < h <k(t+2)+t. \\
            \end{cases} \\
     & \simeq \begin{cases}
            \underbrace{\Ind_r(B_t^m) \ast \dots \ast \Ind_r(B_t^m)}_{m(m^0+m^{t+2}+\dots+m^{(k-1)(t+2)})\text{-copies}},  & \text{if } h=(k-1)(t+2)+t+1, \\
            \underbrace{\Ind_r(B_t^m) \ast \dots \ast \Ind_r(B_t^m)}_{(m^0+m^{t+2}+m^{2(t+2)+\dots+m^{k(t+2)})}\text{-copies}}, & \text{if } h=k(t+2)+t, \\
            \{\text{point}\}, & \text{if }  (k-1)(t+2)+t+1 < h <k(t+2)+t. \\
            \end{cases} \\
    \end{split}
\end{equation*}
Hence, using \cref{lemma:indrbhm limited} and \cref{lemma:join of spheres}, we get the result (recall that $t$ is fixed).
%, {\itshape i.e.},
\vspace*{-0.9cm}\end{itemize}
%
%{\footnotesize \begin{equation*}
%\begin{split}
%    \Ind_r(B_h^m)  & \simeq \begin{cases}
%            \bigvee\limits_{{m^t-1 \choose s}^{m(m^0+m^{t+2}+\dots+m^{(k-1)(t+2)})}} \mathbb{S}^{mr(m^0+m^{t+2}+\dots+m^{(k-1)(t+2)})-1}, & \text{if } h=(k-1)(t+2)+t+1, \\
%             \bigvee\limits_{{m^t-1 \choose s}^{(m^0+m^{t+2}+m^{2(t+2)}+\dots+m^{k(t+2)})}} \mathbb{S}^{r(m^0+m^{t+2}+m^{2(t+2)}+\dots+m^{k(t+2)})-1}, & \text{if } h=k(t+2)+t, \\
%            \{\text{point}\}, & \text{ otherwise}. \\
%            \end{cases}
%    \end{split}
%\end{equation*}}
%This completes the proof of \cref{theorem:indrbhm general}.
\end{proof}

\section{Concluding remarks}\label{sec_end}

In this section, we list a few interesting questions and conjectures. 

\subsection{Universality of higher independence complexes}
It was shown in \cite{EH06} that every simplicial complex arising as the barycentric subdivision of a CW complex may be represented as the $1$-independence complex of a graph. 
One can investigate whether a similar statement holds for for all $r$-independence complexes. 
From the definition it is clear that $\Ind_r(G)$ contains all subsets of $V(G)$ of cardinality at most $r+1$ implying that $\Ind_r(G)$ is always ($r-2$)-connected. 
Moreover, the following example (which was done using SAGE) tells us that the homology groups of $r$-independence complexes of graphs are may have torsion. 
Let $M_s(G)$ denotes the $s^{\text{th}}$ generalized mycielskian of a graph $G$. Then, 
\begin{equation*}
    \tilde{H}_i(\Ind_2(M_4(C_4))) = 
    \begin{cases}
    \Z_2 & \mathrm{~if~} i =3,\\
    \Z^{45} & \mathrm{~if~} i =5,\\
    0 & \mathrm{~otherwise}.\\
    \end{cases}
\end{equation*}

One can now ask the following question.

\begin{question}
Given $r\geq 2$ and an $(r-2)$-connected simplicial complex $X$, does there exists a graph $G$ such that $\Ind_r(G)$ is homeomorphic to $X$?
\end{question}

\subsection{Trees}
Kawamura \cite{Kaw10a} computed the exact homotopy of type of $1$-independence complexes of trees and showed that they are either contractible or homotopy equivalent to a sphere.
In \cref{sec_trees}, it was shown that the homotopy type of higher independence complexes of $m$-ary trees is also a wedge of spheres. 
So, one might hope for a similar result for the class of all trees as well. 

In another project \cite{DSS19} with Samir Shukla, authors have determined the homotopy type of $\Ind_r(G)$ for chordal graphs $G$ (note that class of trees is a subclass of chordal graphs). 
A {\itshape chordal graph} is a graph in which every cycle on more than $3$ vertices has a chord. 
Homotopy type of $1$-independence complexes of chordal graphs was studied by Kawamura in \cite{Kaw10b}. 
Here, we only announce our result, without proving it.

\begin{theorem}[\cite{DSS19}]
The higher independence complexes of chordal graphs are either contractible or homotopy equivalent to a  wedge of spheres.
\end{theorem}

However, the following question is still unanswered.

\begin{question}
Given $r\geq 2$ and a tree $T$, find a formula for the number of spheres in the homotopy decomposition of $\Ind_r(T)$?
\end{question}

\subsection{Shellable higher independence complexes} 
A popular research direction in topological combinatorics is to determine whether the given simplicial complex is shellable or not.
A shellable simplicial complex has the homotopy type of a wedge of spheres.
The results proved in the present article point towards the following question. 
\begin{question}
For which classes of graphs, the higher independence complexes are shellable?
\end{question}

In \cite{wood09}, Woodroofe showed that $1$-independence complexes of chordal graphs are vertex-decomposable (hence shellable \cite[Theorem 1.2]{TVH08}). 
Based on our computations of $\Ind_r(G)$ for various chordal graphs and $r\geq 2$, using Macaulay2, we expect a similar result for $r$-independence complexes of chordal graphs. 
% \begin{question}
% Whether $\Ind_r(G)$ is vertex-decomposable for each $r\geq 2$ and chordal graph $G$? 
% \end{question}

% There is also the case of chordal graphs. 

\begin{conjecture}
If $G$ is a chordal graph then the complex $\Ind_r(G)$ is vertex-decomposable for each $r\geq 2$.
\end{conjecture}

\subsection{Grid graphs} 
For $m,n \geq 2$, a rectangular {\itshape grid graph}, denoted $G_{m,n}$ is a graph with $V(G_{m,n}) = \{(i,j): i \in [m],~j\in [n]\}$ as its vertex set and $(i,j)$ is adjacent to $(i_1,j_1)$ in $G_{m,n}$ if and only if either `$i_1=i$ and $j_1=j+1$' or `$j_1=j$ and $i_1=i+1$'. 
In the last decade, $1$-independence complexes of grid graphs were studied in detail (see \cite{BLN08, BH17, Jon10}). 
We have analysed the complex $\Ind_r(G_{2,n})$ (for small values of $n$) and also computed their homology groups of using SageMATH \cite{Sage} (see \cref{table:grid graph 2n} below). 
Based on our calculations, we make the following conjecture.

\begin{conjecture}\label{conj:grid graphs}
For all $r\geq n$, $\Ind_r(G_{2,n})$ is either contractible  or homotopy equivalent to a wedge of spheres of dimension $r-1$. 
\end{conjecture}

From \cref{table:grid graph 2n} (used SageMath \cite{Sage} for these computations), we also see that $\tilde{H}_i(G_{2,9})$ is non-trivial in two different dimensions (the notation $i:\Z^p$ means $\tilde{H}_i(\Ind_r(G_{2,n}))=\Z^p$).
This raises the following question.

\begin{question}
What is the homotopy type of higher independence complexes of grid graphs $G_{m,n}$?
\end{question}

\begin{table}[H]
	\centering
	\begin{tabular}{ |c|c|c|c|c|c|c|c|c|c| }
		\hline
		\backslashbox{$n$}{$r$} &{$1$} &{$2$} & {$3$} &{$4$}  &{$5$}& {$6$} &{$7$}& {$8$} & {$9$} \\ 
		\hline \hline &&&&&&&&& \\[-1em]
		
		{$1$} &{$0:\Z$} & $0$ &{$0$}& {$0$} &$0$  &{$0$}& $0$ &{$0$} & {$0$} \\ 
		\hline &&&&&&&&& \\[-1em]
		
		{$2$} &{$0:\Z$} & $1:\Z^3$ &{$2:\Z$}& {$0$} &$0$  &{$0$}& $0$ &{$0$} & {$0$} \\ 
		\hline &&&&&&&&& \\[-1em]
		
		{$3$} &{$1:\Z$} & {$1:\Z$} & {$2:\Z^5$} & {$3:\Z^5$} & {$4:\Z$} & $0$ &{$0$}& {$0$} & {$0$} \\ 	\hline &&&&&&&&& \\[-1em]
		
		{$4$} &{$1:\Z$} & {$3:\Z^2$} & {$0$} & {$3:\Z^7$} & {$4:\Z^{13}$} & $5:\Z^{7}$ &{$6:\Z$}& {$0$} & {$0$} \\ 	\hline &&&&&&&&& \\[-1em]
	
		{$5$} &{$2:\Z$} & {$3:\Z^7$} & {$5:\Z$} & {$0$} & {$4:\Z^{8}$} & $5:\Z^{25}$ &{$6:\Z^{25}$}& {$7: \Z^9$} & {$8:\Z$} \\ \hline &&&&&&&&& \\[-1em]
	
	    {$6$} &{$2:\Z$} & {$3:\Z$} & {$5:\Z^{17}$} & {$7:\Z^2$} & {$0$} & $5:\Z^{8}$ &{$6:\Z^{40}$}& {$7: \Z^{63}$} & {$8:\Z^{41}$} \\ \hline &&&&&&&&& \\[-1em]
	
	    {$7$} &{$3:\Z$} & {$5:\Z^{10}$} & {$5:\Z^{8}$} & {$7:\Z^{31}$} & {$9:\Z$} & $0$ &{$6:\Z^{8}$}& {$7: \Z^{56}$} & {$8:\Z^{128}$} \\ \hline  &&&&&&&&& \\[-1em]
	
	    {$8$} &{$3:\Z$} & {$5:\Z^{13}$} & {$8:\Z$} & {$7:\Z^{49}$} & {$9:\Z^{57}$} & $11:\Z^2$ &{$0$}& {$7: \Z^{8}$} & {$8:\Z^{72}$} \\ \hline  &&&&&&&&& \\[-1em]
	
	    {$9$} &{$4:\Z$} & {$5:\Z;~ 7:\Z^4$} & {$8:\Z^{45}$} & {$7:\Z^{8}$} & {$9:\Z^{160}$} & {$11:\Z^{79}$}& $13:\Z$ & {$0$} & {$8:\Z^8$} \\ \hline
	\end{tabular}
	\caption{Reduced homology groups of $r$-independence complexes of grid graphs $G_{2,n}$. For all $n \leq 9$ and $r\leq 9$, $i:0$ ( {\itshape i.e.} $\tilde{H}_i(\Ind_r(G_{2,n}))=0$) for all $i$ not mentioned in the table.}\label{table:grid graph 2n}
\end{table}

\section*{Acknowledgements}
The authors would like to thank the anonymous referee for careful reading and helpful suggestions. The authors are partially funded by a grant from Infosys Foundation. PD is also partially funded by the
MATRICS grant MTR/2017/000239. 

%\printbibliography
\bibliographystyle{abbrv}

%\bibliography{ref}
\end{document}